\title{On Chasles' Quadrilateral Theorem \\
}
\author{ Leah Wrenn Berman and Jürgen Richter-Gebert }

\documentclass{article}
\usepackage{graphicx}
\usepackage{amsthm}

\newtheorem*{theorem*}{Theorem}
\newtheorem*{theoremA*}{Theorem A}
\newtheorem*{theoremB*}{Theorem B}
\newtheorem*{theoremC*}{Theorem C}

\newtheorem{theorem}{Theorem}
\newtheorem{notheorem}{Not--a--Theorem}

\newtheorem{remark}{Remark}

\usepackage{amssymb}
\usepackage[font=small,labelfont=bf]{caption} 
\usepackage{color}
\usepackage{mathtools}
\usepackage{amsmath}
\usepackage[dvipsnames]{xcolor}
\usepackage[all]{nowidow}
\usepackage{float}

\usepackage{soul}
\usepackage[normalem]{ulem} 

\usepackage[color links=true, backref=page]{hyperref} 

\hypersetup{
  colorlinks,
  citecolor=blue,
  linkcolor=Red,
  urlcolor=Blue
  }
\def\upto{, \ldots ,}

\begin{document}
\newpage
\maketitle

\def\upto{, \ldots ,}
\def\lines{\square}
\def\points{\bigcirc}
\def\lines{\vee}
\def\join{\vee}
\def\points{\wedge}
\def\meet{\wedge}
\parskip=1mm
\setlength{\parindent}{.5cm}

\begin{abstract}
Chasles’ Quadrilateral Theorem \cite{Cha1843, Cha1860}  is a classical statement about four tangents to a conic that simultaneously circumscribe a circle. In its various formulations, it relates the concurrence of certain lines to the existence of confocal conics or inscribed circles.

We show that several classical and modern versions of this theorem are affected by subtle ambiguities arising from multiple solutions in the underlying geometric constructions. 
These ambiguities often enter through seemingly natural extensions of otherwise correct statements.

We provide a systematic analysis of these issues and present coherent formulations of the theorem that avoid these inconsistencies. In particular, we interpret the theorem in a projective framework and relate it to the Cayley–Bacharach theorem, which explains the underlying incidence structure.


\end{abstract}

\section{Chasles' Quadrilateral Theorem}
Sometimes in mathematics, statements are more subtle than one might initially expect.
This happened to us when we needed a suitable version of 
a classical theorem of Chasles about a quadrilateral whose lines are simultaneously tangent to a conic and a circle. We will refer to the theorem as
{\it Chasles' Quadrilateral Theorem} (CQT) for short.\footnote{
The terminology surrounding this theorem is somewhat confused in the modern literature. The classical result appears first in a short note of Chasles from 1843 \cite{Cha1843}, and historically the correct name would simply be `Chasles' Theorem'. However, Michel Chasles pointed out many important theorems. So this would not be distinctive enough. In the literature one finds two double attributions to CQT. One of them is ``Graves--Chasles''. It  seems to originate from a stylistic remark by Darboux, who occasionally referred to it as a ``propri\'{e}t\'{e} de Graves--Chasles''; however, Graves himself merely translated an earlier (1841) \cite{ChaGra1841} paper of Chasles that does not contain the tangential quadrilateral result. A different double naming, ``Chasles--Reye'', used for example by Izmestiev and Tabachnikov \cite{IzTa17}, perhaps has more justification, since Reye appears to have presented  the first correct formulation and  proof of the statement \cite{Rey1896}. Still, to avoid both historical inaccuracies and terminological ambiguity, we refer to the result simply as  Chasles' Quadrilateral Theorem~(CQT).
} 
We needed this statement in the context of
our work relating $(n_4)$ configurations to Poncelet's porism  \cite{BGRGT24a}. There, we prove the movability of large classes of 
$(n_4)$-configurations. An $(n_{4})$ (geometric) configuration is a collection of $n$ points and $n$ (straight) lines such that every point has four lines passing through it and every line is incident to four points. In the technical core of that work, we needed a version of {\it Chasles' Quadrilateral Theorem} that proves that four specific lines in our construction (that arise from a \emph{Poncelet Grid})
meet in a single point. At first, we were convinced that it should be easy to just quote Chasles' Quadrilateral Theorem from the literature, since CQT is a very classical theorem  dating back to the work of the 19th century geometers.
However, in our investigation of the literature, we realized that all of the versions that we found fell in one of the following two categories: Either they were limited to a scenario that was more restricted than the one we needed, or they \emph{claimed} what we need but their statement was flawed, asserting concurrence of lines under certain conditions for which we could find explicit counterexamples. These counterexamples have their origin in the occurrence of additional solutions that arise under certain conditions.

The aim of this article is to clarify Chasles' Quadrilateral Theorem and its various classical formulations, with particular emphasis on the subtle ambiguities that arise in their interpretation. We analyse how different distributions of hypotheses and conclusions affect the validity of the statement, and we provide a coherent framework that resolves these issues.

In what follows, we first encounter the different flavours of the theorem (projective, primal, dual, and Euclidean) and their ramifications. In particular, we show some of the corollaries that have been drawn from it and point to some places in which flawed extensions of the CQT are presented. Those flawed versions even date back to Chasles himself \cite{Cha1860} from 1860 and to the treatment of this statement by Darboux in 1917 \cite{Dar17}. However, several modern treatments of the matter also suffer from the same subtle difficulties (e.g., \cite{AkBo18, Iz19, StaB}). In Section 3, we give a detailed analysis of the theorem and its projective nature and relate it to the well-known Cayley--Bacharach Theorem (which is also closely related to Chasles, who discovered its first instance). We also present an approach following Chasles' original line of reasoning. In particular, we explain how the flawed formulations come from neglected multiple solutions that arise in the realization space of the core theorem, which may be easily overlooked. 

\medskip

\subsection{Pitfalls of Chasles' Quadrilateral Theorem }

In its purest form, Chasles' Quadrilateral  Theorem can be stated as follows (see Figure~\ref{fig:pure}).

\begin{theorem}\label{CGTpure}
{\bf (pure version of the CQT):} Let $\mathcal{X}$ be an ellipse in the Euclidean plane and let $a,b,c,d$ be four distinct tangents to $\mathcal{X}$. Then the following statements are equivalent.
\begin{itemize}
\item[(i)] $a,b,c,d$ are tangent to a circle $\mathcal{C}$.
\item[(ii)] The intersections  $P=a\meet b$ and $Q=c\meet d$ lie on a conic $\mathcal{R}$ confocal to $\mathcal{X}$.
\end{itemize}
\end{theorem}

Stated in this way, the theorem is a statement of Euclidean geometry, since it explicitly refers
to a {\it circle} and {\it confocal conics}. Nevertheless, the core of this statement is of a projective nature, which we will later investigate more closely. For the moment, we will stick to the Euclidean version of the statement and elaborate on the configuration of geometric objects that we might associate 
with~it.
 \begin{figure}[ht]
\centering
\includegraphics[width=.45\textwidth]{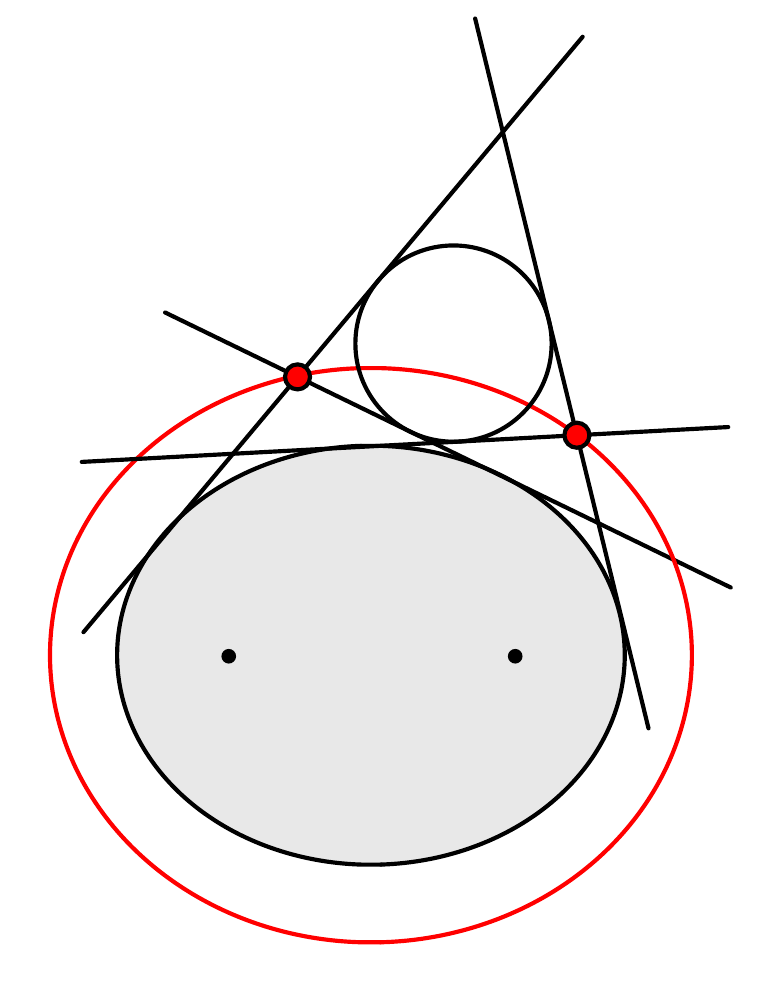}\qquad\quad
\includegraphics[width=.45\textwidth]{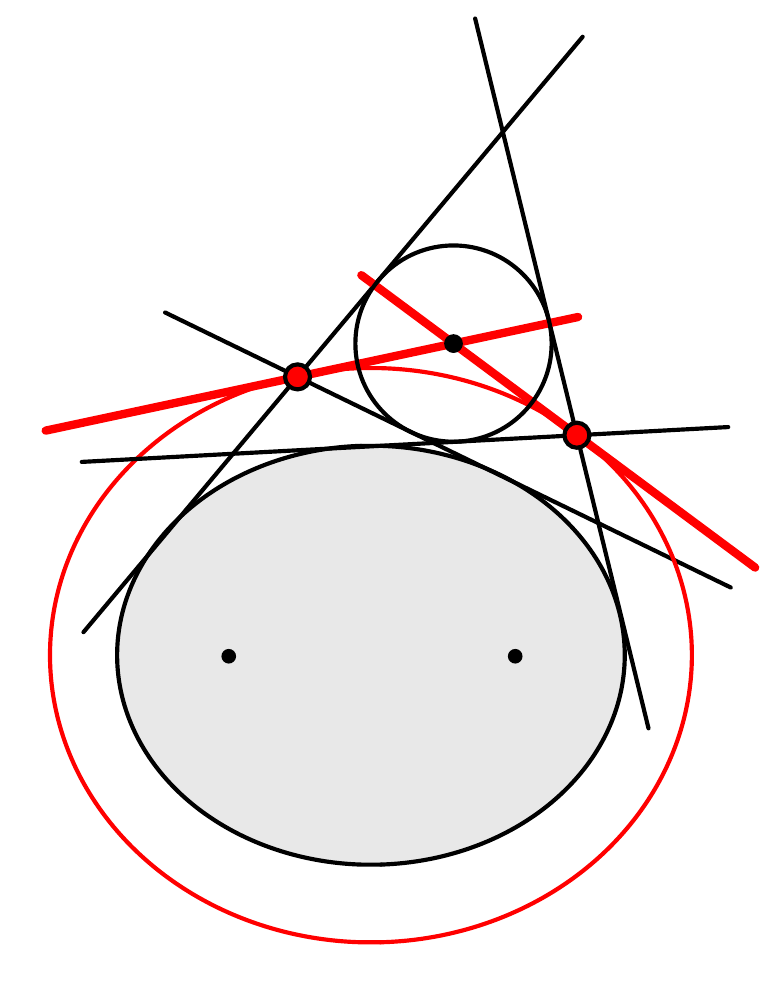}
\begin{picture}(0,0)
\put(-310,28){\footnotesize {$\mathcal{X}$}}
\put(-320,16){\footnotesize {$\mathcal{R}$}}
\put(-248,153){\footnotesize {$\mathcal{C}$}}
\put(-290,130){\footnotesize {${P}$}}
\put(-222,118){\footnotesize {${Q}$}}
\put(-330,65){\footnotesize {${a}$}}
\put(-330,110){\footnotesize {${c}$}}
\put(-310,140){\footnotesize {${b}$}}
\put(-254,195){\footnotesize {${d}$}}

\put(-125,28){\footnotesize {$\mathcal{X}$}}
\put(-135,16){\footnotesize {$\mathcal{R}$}}
\put(-63,153){\footnotesize {$\mathcal{C}$}}
\put(-105,130){\footnotesize {${P}$}}
\put(-37,118){\footnotesize {${Q}$}}
\put(-145,65){\footnotesize {${a}$}}
\put(-145,100){\footnotesize {${c}$}}
\put(-125,140){\footnotesize {${b}$}}
\put(-69,195){\footnotesize {${d}$}}
\put(-68,135){\footnotesize {${M}$}}
\put(-150,117){\footnotesize {${t_P}$}}
\put(-10,92){\footnotesize {${t_Q}$}}
\end{picture}
\caption{The Chasles' Quadrilateral  Theorem in its purest form (left) and its extended form (right), including tangents to $\mathcal{R}$ at $P$ and $Q$.}	
\label{fig:pure}
\end{figure}
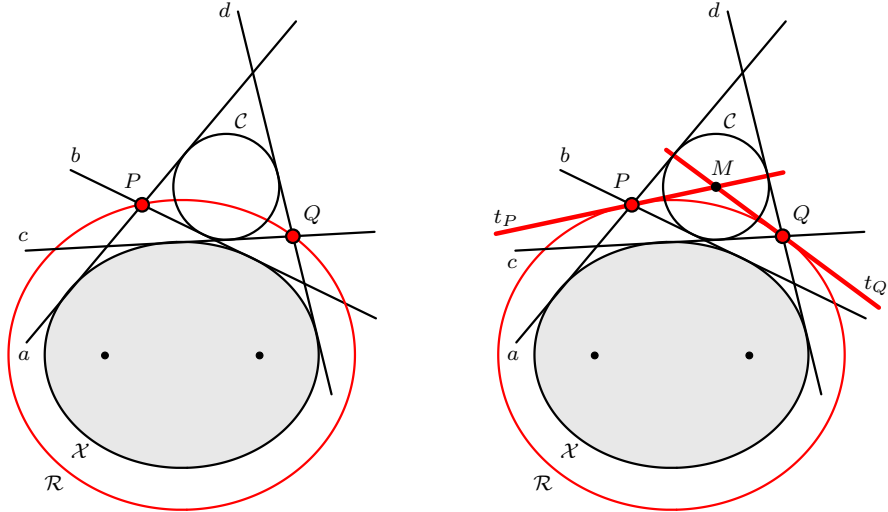

There are two important aspects that lead to a natural  extension of this configuration
and the Theorem (which is frequently done in the literature).
Consider the image in Figure~\ref{fig:pure} on the right. 
There, the two tangents at the points $P$ and $Q$ to the conic $\mathcal{R}$ are added to the configuration.
In fact, the  tangent at $P$ (resp. $Q$) is an angle bisector of the lines $a$ and  $b$ (resp.  $c$ and $d$).
The tangents at  $P$ and $Q$  meet at a point $M$ that turns out to be the center of the circle.
The fact that the tangent at point $P$ to $\mathcal{R}$ is the angle bisector of the two lines $a$ and $b$ under the condition that $a$ and $b$ are tangent to an ellipse  $\mathcal{X}$ confocal to $\mathcal{R}$ 
is a well-known fact in the geometry of billiards on elliptic tables (see for instance \cite{DraRa11}, Proposition~2.4).

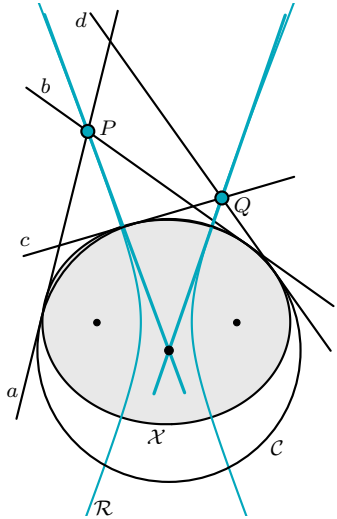
\begin{figure}[ht]
\centering
\includegraphics[width=.45\textwidth]{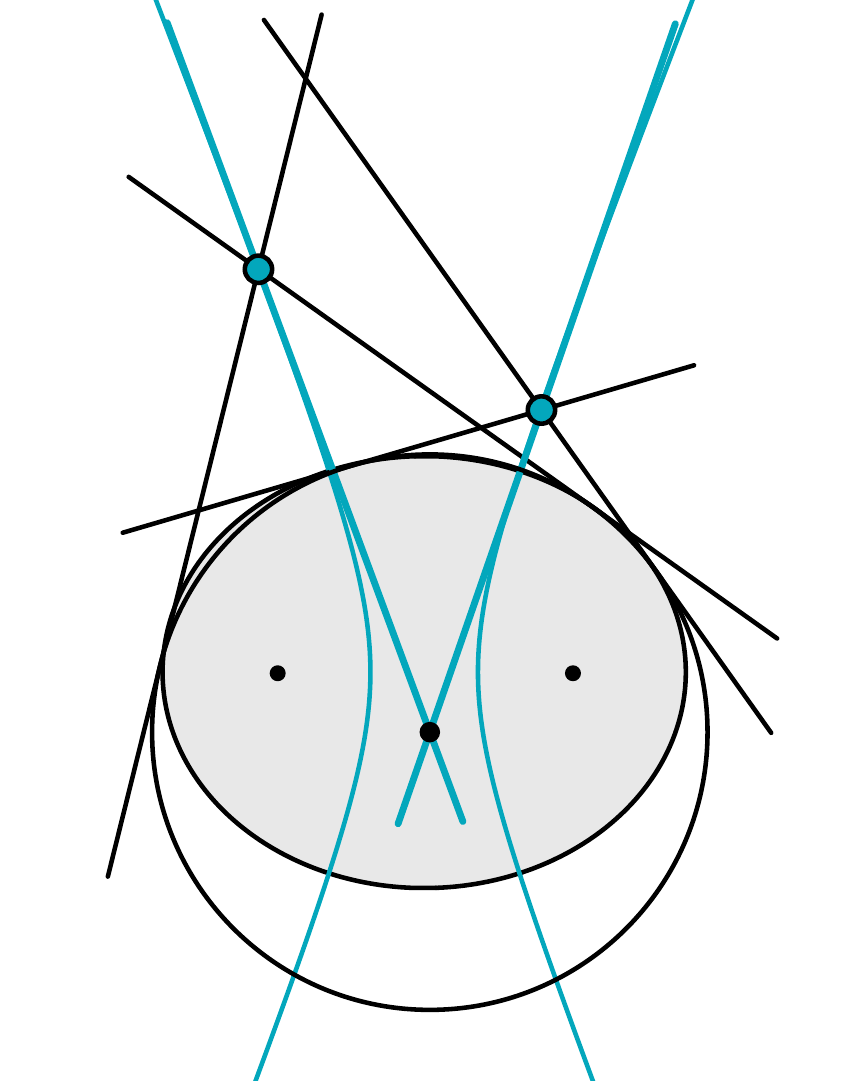}
\begin{picture}(0,0)
\put(-90,27){\footnotesize {$\mathcal{X}$}}
\put(-110,0){\footnotesize {$\mathcal{R}$}}
\put(-43,23){\footnotesize {$\mathcal{C}$}}
\put(-108,144){\footnotesize {${P}$}}
\put(-57,115){\footnotesize {${Q}$}}
\put(-143,45){\footnotesize {${a}$}}
\put(-138,102){\footnotesize {${c}$}}
\put(-130,160){\footnotesize {${b}$}}
\put(-117,185){\footnotesize {${d}$}}

\end{picture}
\caption{The Chasles' Quadrilateral  Theorem in a generic situation with a hyperbola as $\mathcal{R}$.}	
\label{fig:pure_H}
\end{figure}

However, there is an important conceptual difference between the angle bisector of the lines $a$ and $b$ and the tangent to $\mathcal{R}$ at $P$.
While there is just {\it one} tangent  at $P$, there are ${\it two}$ angle bisectors of $a$ and $b$: one is shown in Figure~\ref{fig:pure},
but the line through $P$ orthogonal to the first angle bisector is also an angle bisector.
This ambiguity is one of the sources of confusion and easily  leads to faulty extensions of Theorem~\ref{CGTpure}.
The ambiguity of angle bisectors is complemented by an ambiguity in the choice of a  conic through $P$ confocal to $\mathcal{X}$. As a matter of fact, there are two such conics: an ellipse (shown in Figure \ref{fig:pure}) and a hyperbola (whose tangent is the other angle bisector). See the Figure \ref{fig:pure_H} for an instance where the conic $\mathcal{R}$ is a hyperbola rather than an ellipse.

 This ambiguity in the choice of conics is the second source of confusion, as we will see. 
 The  situations shown in Figure \ref{fig:pure} and Figure \ref{fig:pure_H} usually do not coexist for  generic choices of $P$ and $Q$. However, if $P$ and $Q$ are symmetric with respect to the perpendicular bisector of the foci of $\mathcal{X}$ (in our picture, the vertical symmetry axis), both situations may be present simultaneously. That is, both a confocal ellipse and a confocal hyperbola can be drawn through $P$ and $Q$. Corresponding situations are shown in Figure~\ref{fig:flaw}.

Let us contrast our Theorem 1 with a version that can be found in the classical and extremely important book {\it Principes de géométrie analytique} by Gaston Darboux from 1917 \cite{Dar17}: 
\medskip

\begin{center}\includegraphics[width=.95\textwidth]{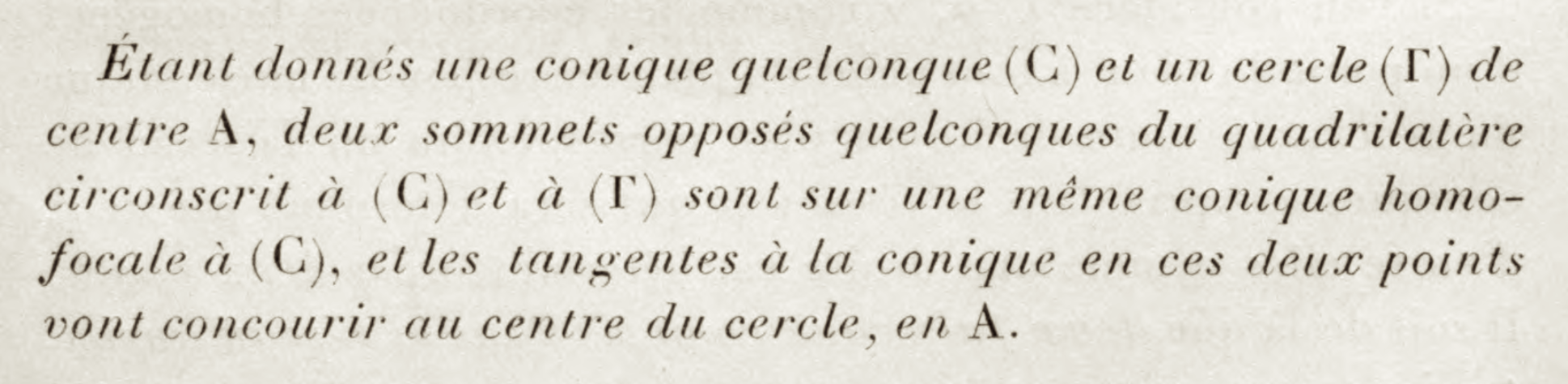}\end{center}
\medskip

\noindent
The first part of this statement describes one part of the above equivalence. In translation (and with the names of the objects adapted to those used in this article) this reads:

\noindent
\begin{quote}{\it Let $\mathcal{X}$ be any conic and $\mathcal{C}$ be any circle with centre $M$. Then any two opposite points of
the quadrilateral circumscribed at both   $\mathcal{X}$  and  $\mathcal{C}$ lie on a conic  $\mathcal{R}$ confocal to $\mathcal{X}$, {\color{red} \ul{and the tangents at these points will meet in the centre of the circle.}}
}\end{quote}

The drawing that can be often found along with this statement is given in  Figure~\ref{fig:pure} on the right. The two red lines in that figure are the tangents mentioned in this statement.
 The core part of this statement is correct; however the extension {\it ``$\ldots$and the tangents$\ldots$''} underlined and marked in red suffers from a significant problem. 

The problem arises in the following way. The first part of this statement describes a configuration consisting of a conic and a circle that are simultaneously tangent to four given lines. Implicitly,  this formulation suggests that if the conic and the four lines are given, the choice of the circle is unique. This is true in general. However, it becomes false in specific situations.  There are situations in which there is a second circle that is simultaneously tangent to all four lines, when the configuration is symmetric with respect to the perpendicular bisector of the foci of the conic.
In this case, neither the conic $\mathcal{R}$ nor the circle $\mathcal{C}$ are uniquely determined. The situation is clarified in Figure~\ref{fig:flaw}. The first picture in this series shows a situation in which the four lines and the conic $\mathcal{X}$ are symmetric with respect to the vertical axis. In this case, there are two conics  $\mathcal{R}$ and  $\mathcal{R'}$ confocal to  $\mathcal{X}$ through $P$ and $Q$, and also two circles  $\mathcal{C}$  and  $\mathcal{C'}$ tangent to the four black lines. Depending on the chosen combination, they may be consistent with the tangent construction or not. The two middle pictures show the two consistent choices while the last picture shows an inconsistent choice in which the intersection of the tangents is not the center of the circle. In other words: The rightmost picture of Figure~\ref{fig:flaw} is a counterexample to Darboux's statement.

\begin{figure}[ht]
\centering
\includegraphics[width=.24\textwidth]{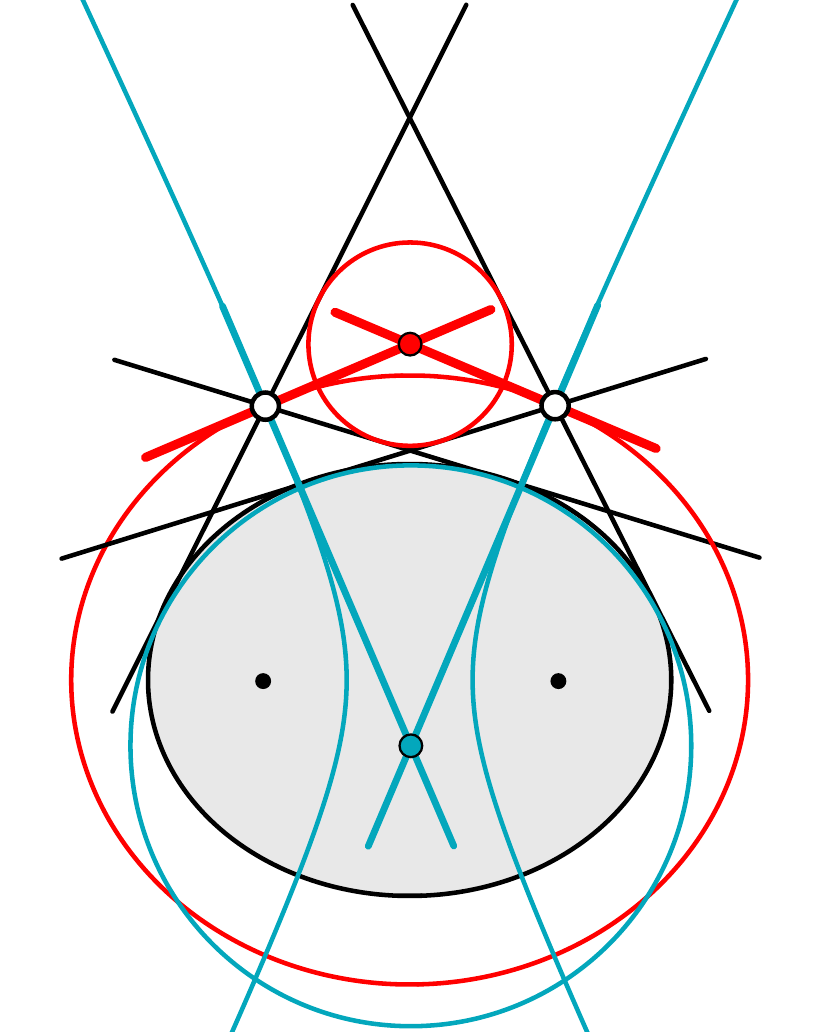}
\includegraphics[width=.24\textwidth]{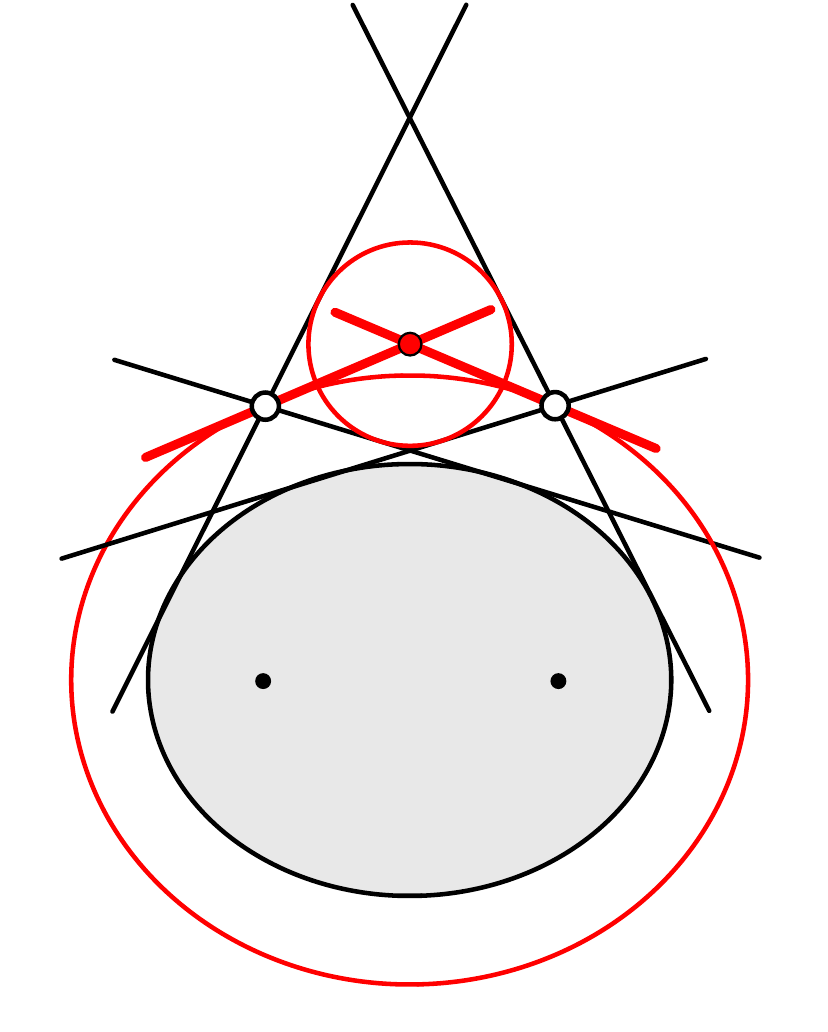} 
\includegraphics[width=.24\textwidth]{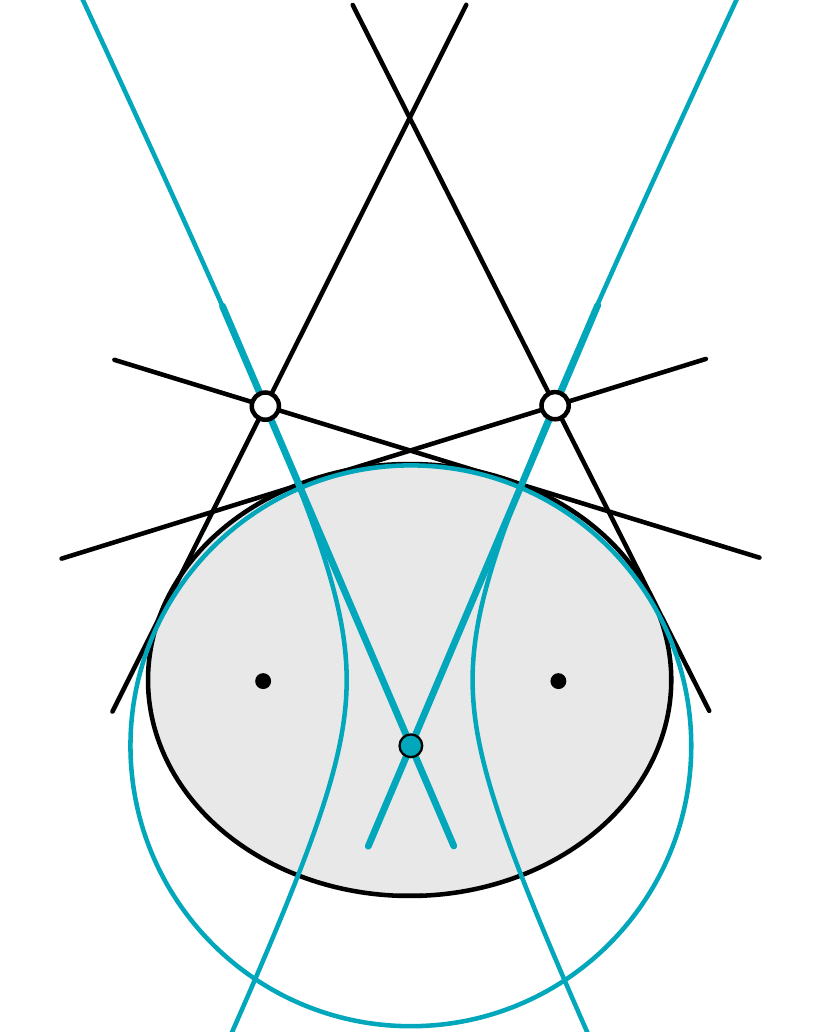} 
\includegraphics[width=.24\textwidth]{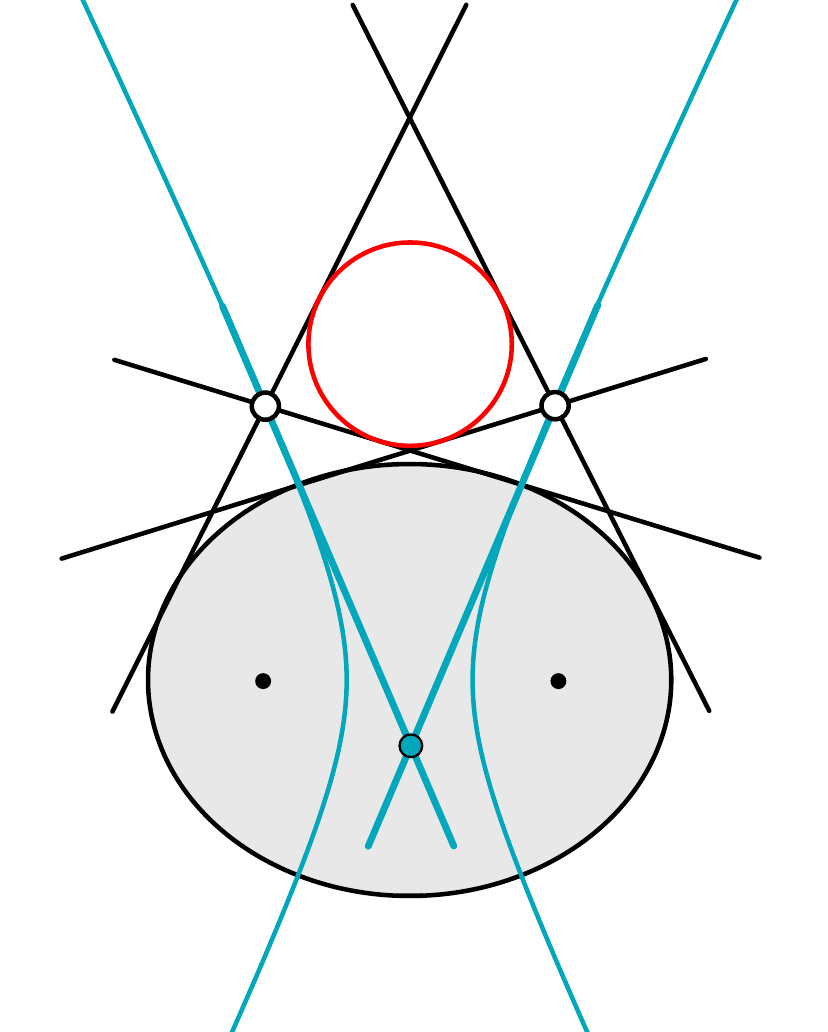} 
\begin{picture}(0,0)
\end{picture}
\caption{
The symmetric situation with the two possible choices for the conic $\mathcal{R}$ and the circle $\mathcal{C}$ (leftmost).
The following images show two consistent choices  (middle) and one inconsistent choice (rightmost).
}	
\label{fig:flaw}
\end{figure}

One might be tempted to assume that the flawed formulation is a singular appearance in the literature, but unfortunately this is not the case. Similar situations arise both in older and in more modern sources. Chasles wrote two articles about a cycle of theorems  
around closely related matters,
one in 1843 \cite{Cha1843} and one in 1860 \cite{Cha1860}. There, he first develops a basis of four general theorems and then derives as many as 34 corollaries from them
(none of them with an explicit proof).  Several of these results suffer from the same problems created by ambiguous intermediate results of the intermediate geometric constructions that may lead to inconsistent instances.
We here  focus on Chasles' Corollary~20~\cite{Cha1860} (which is the original form of what we call  Chasles' Quadrilateral Theorem). It is the direct counterpart of Darboux's statement,  which is frequently used in modern literature on elliptical billiards, integrable systems, and discrete differential geometry:

\medskip

\begin{center}\includegraphics[width=.95\textwidth]{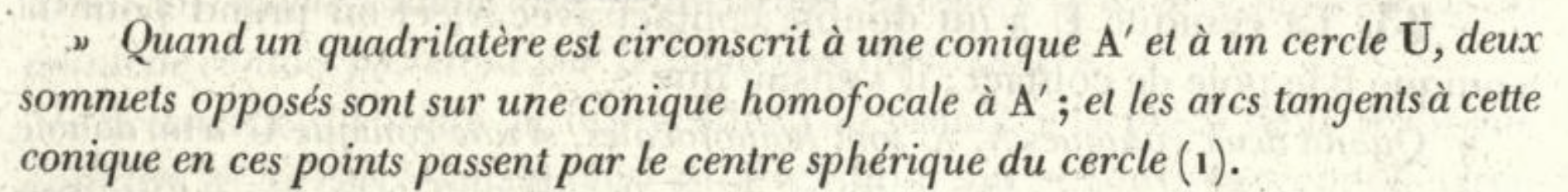}\end{center}
\medskip

\noindent
Literally translated, we get a paraphrase of Darboux's statement:
\noindent
\begin{quote}{\it When a quadrilateral is circumscribed by a conic $\mathcal{X}$  and a circle  $\mathcal{C}$, two opposite vertices are on a conic
confocal to  $\mathcal{R}$; {\color{red}\ul{and the arcs tangent to this conic at these points pass through the spherical center of the circle}.}
}\end{quote}

Before we analyse the deeper reasons why these theorems lead to such ambiguities, we quote two relatively modern sources that suffer from similar problems. For reasons that will become apparent later, we will call the statements of Chasles and Darboux {\it dual} formulations. Similar problems may also arise in the {\it primal} world where the roles of points and lines are interchanged. An instance can be found in an article by  Akopyan and Bobenko  \cite{AkBo18} in which  their Lemma~2.3 states the following:

\begin{quote}{\it Let $a, b, c, d$ be four points on a conic $\alpha$ and the lines $(ab)$ and $(cd)$
touch some other conic $\beta$. Then lines $(bd)$ and $(ac)$ touch some conic $\gamma$ from the
pencil generated by the conics $\alpha$ and $\beta$. {\color{red}\ul{Moreover, the tangent points of $\beta$ with $(ab)$
and $(cd)$ and the tangent points of $\gamma$ with $(bd)$ and $(ac)$ are collinear.}}
}\end{quote}

Again the problem lies in the `moreover'-part. Figure~\ref{fig:BobAk} on the left gives the instance that was intended by the authors (this situation will arise in the generic case). The small black dots indicate that all conics are taken from the same pencil. The bigger black dots are the ones about which the collinearity is stated. On the right, the figure shows an instance in which all hypotheses of the statement hold but the collinearity claimed in the red part of the statement fails.
Instead, the  tangent points lie on another conic (shown thick and blue) that is also tangent to the lines at the touching points. The reason for the specific coloring of the elements will become transparent later.
\begin{figure}[ht]
\centering
\includegraphics[width=.45\textwidth]{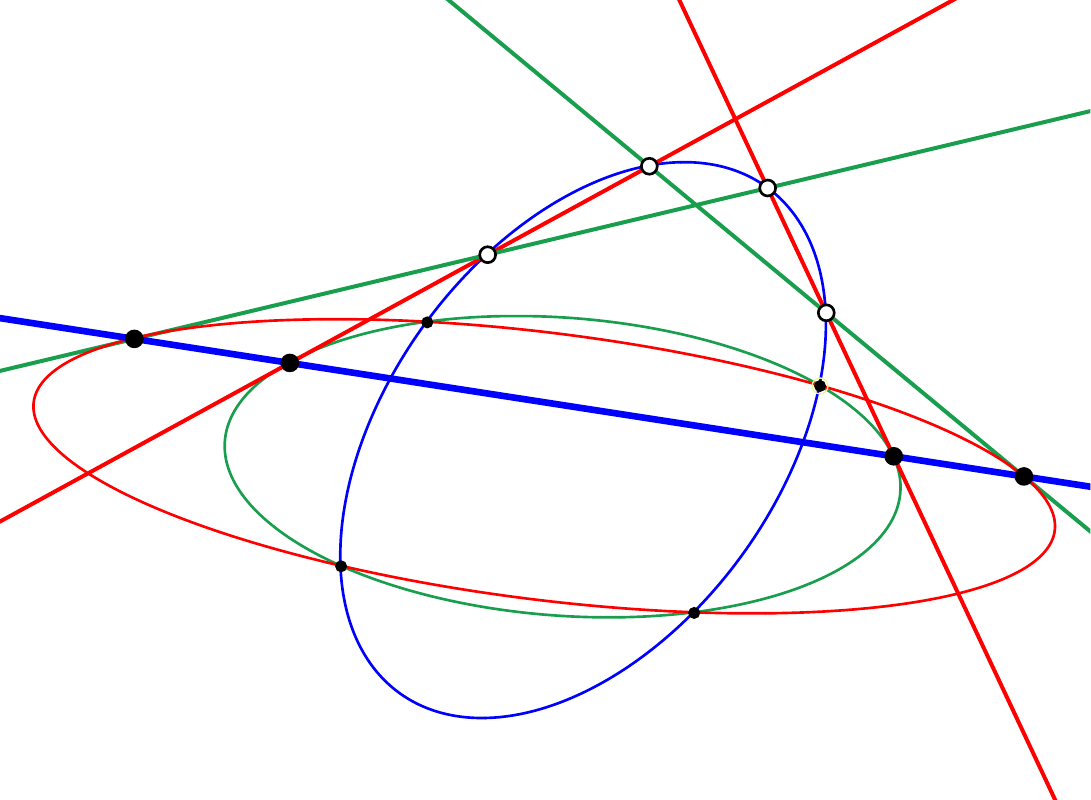}\hfill
\includegraphics[width=.45\textwidth]{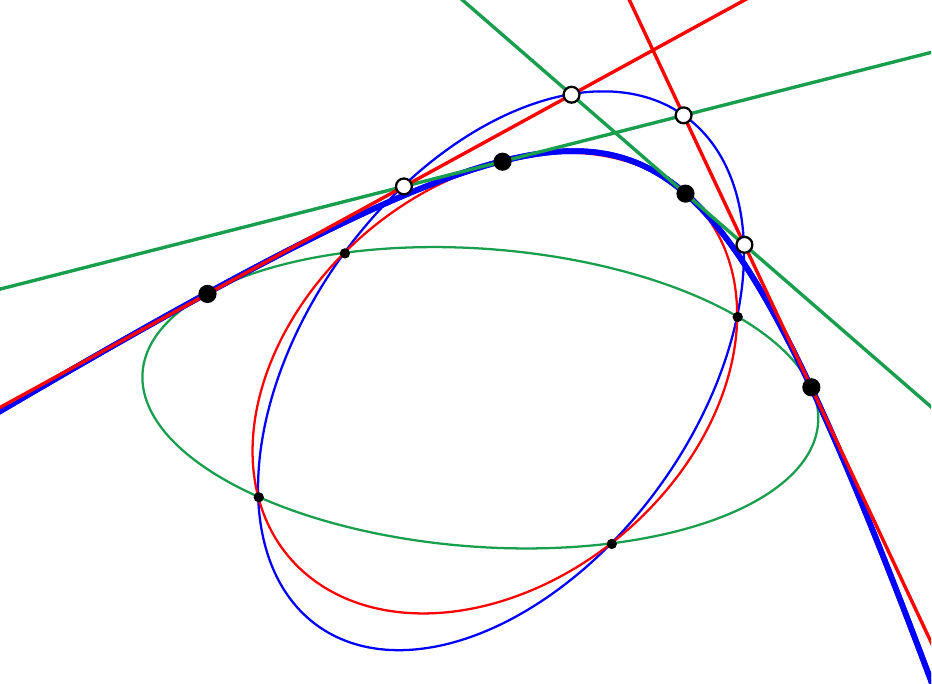}
\begin{picture}(0,0)
\put(-265,10){\footnotesize {$\alpha$}}
\put(-75,5){\footnotesize {$\alpha$}}
\put(-282,80){\footnotesize {$a$}}
\put(-97,87){\footnotesize {$a$}}
\put(-254,95){\footnotesize {$b$}}
\put(-65,103){\footnotesize {$b$}}
\put(-235,90){\footnotesize {$c$}}
\put(-44,97){\footnotesize {$c$}}
\put(-225,70){\footnotesize {$d$}}
\put(-33,75){\footnotesize {$d$}}
\put(-85,25){\footnotesize {$\beta$}}
\put(-275,20){\footnotesize {$\beta$}}
\put(-325,35){\footnotesize {$\gamma$}}
\put(-95,15){\footnotesize {$\gamma$}}

\end{picture}
\caption{The intended geometric situation in \cite{AkBo18} Lemma 2.5 (left) and an explicit counterexample to the statement (right).}	
\label{fig:BobAk}
\end{figure}
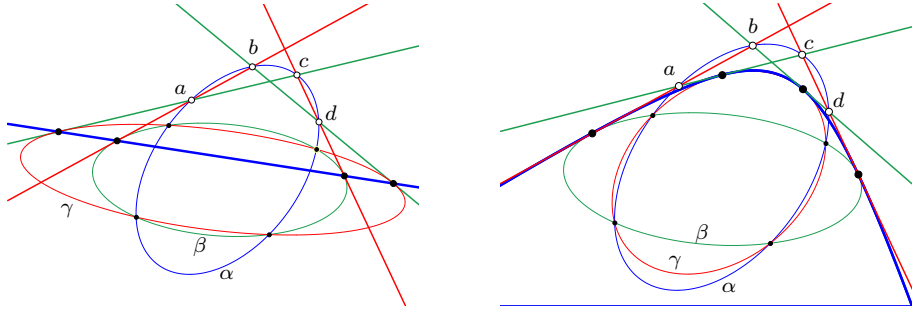

In their article,  Akopyan and Bobenko  mention that their Lemma is a paraphrased version of 
statements in \cite{Ber87} and \cite{AlkZas07}. The statements in those two sources are correct; however, the glitch slipped into the statement during paraphrasing it. It should be mentioned that from their Lemma 2.3 they derive a correct Version of the CQT (their Theorem 2.5).

We also mention that  the articles \cite{StaB} by Stachel (Theorem 3.5)  and 
\cite{Iz19} by Izmestiev (Theorem 2.12) contain
 similar flawed versions (this time again in the dual scenario). Since the situation is structurally similar to the ones we have shown so far, we omit a detailed analysis here.

\medskip

There is a good reason for the flaws in the statement to happen. If one considers the space of possible instances of the hypotheses of the statements, the possibility for the ambiguity to happen arises only on a (hard-to-detect) low dimensional submanifold of the configuration space. Thus, every generic drawing will be in a region that does not exhibit the ambiguity. As a matter of fact, the flawed instances were found by experimenting with the dynamic geometry software Cinderella \cite{RiKo99}, in which it is relatively easy to explore configuration spaces of incidence theorems.

\medskip


We now present a version of the CQT that is logically correct and involves the statement about the tangents:

\begin{theorem} {\bf (CQT with tangents)}:\label{CGTtangent}
Let $\mathcal{X}$ be an ellipse in the Euclidean plane and let $\mathcal{R}$ be a conic confocal to  $\mathcal{X}$. Let $P$ and $Q$ be two distinct points on the conic  $\mathcal{R}$
and let $t_P$ and $t_Q$ be the two tangents to that conic through $P$ and $Q$, respectively. Then the two tangents $t_P$ and $t_Q$ intersect in a point $M$. The point $M$ is the center of a circle $\mathcal{C}$ touching the four tangents $a,b,c,d$ from $P$ and $Q$ to $\mathcal{X}$.

\smallskip
Conversely, consider four tangents $a,b,c,d$ that are simultaneously tangent to $\mathcal{X}$ and to a circle $\mathcal{C}$.
Let $M$ be the center of the circle and let  $P=a\meet b$ and $Q=c\meet d$. Then there is a conic $\mathcal{R}$ confocal to $\mathcal{X}$ 
and tangent to the lines $M\join P$ and  $M\join Q$ at
$P$ and $Q$.
\end{theorem}
\smallskip

Theorem \ref{CGTtangent} will serve as our main reference formulation. Most subsequent versions will be derived from or are equivalent to this statement.
We will present several proofs for this statement and (more general) variants of this theorem in the subsequent sections.

A few remarks are appropriate here. We intentionally split the {\it equivalence} statement in two parts, explicitly showing  how to construct a corresponding circle $\mathcal{C}$ from  $\mathcal{R}$ and vice versa. Thus  $\mathcal{C}$ is constructed as {\it the} circle that is compatible with the choice of  $\mathcal{R}$.
Secondly, we might explicitly have to take care of  limit situations. In our statement, the conic $\mathcal{R}$ could also be a hyperbola intersecting the conic $\mathcal{X}$. Thus it may happen that, for instance, point $P$  lies at the intersection of 
 $\mathcal{R}$ and  $\mathcal{X}$. An examination of the limiting case shows that in this situation the circle should 
 be tangent to both identical lines $a$ and $b$ and at the same time must contain the point $P$. 
 The situation may even be more drastic when $\mathcal{R}$ is a hyperbola. The point $P$ (or $Q$) may be chosen inside the ellipse $\mathcal{X}$. In that case the tangents $a$ and $b$ will turn out to be complex. Nevertheless they still will touch the circle, if we consider {\it touching} as a purely algebraic condition. 

 \smallskip

The last considerations show that if we want to fully understand Chasles' Quadrilateral  Theorem we have to take complex objects into account as well. This is one of the reasons why versions of Theorem~\ref{CGTpure} (or~\ref{CGTtangent}) are usually stated for ellipses only.  We also attempt to cover the general complexified case.
  \medskip

There is another extension of Theorem~\ref{CGTpure} that immediately comes into play when one considers the fact that the lines $a,b,c,d$ play completely symmetric roles. Considering the intersection figure of these six lines we get three pairs of points that could each play the role of $P$ and $Q$. So let $S=a\meet c$, $T=b\meet d$ and let
$U=b\meet c$, $V=a\meet d$. We get (as a direct consequence of Theorem~\ref{CGTpure}):

\begin{theorem} {\bf (CQT with three conics)}: \label{CGTfour}
Let $\mathcal{X}$ be an ellipse in the Euclidean plane. Let $a,b,c,d$ be four distinct tangents to $\mathcal{X}$
and let $P,Q; S,T; U,V$  be the three pairs of intersections as described above. Then the following statements are equivalent.
\begin{itemize}
\item[(i)] $a,b,c,d$ are tangent to a circle $\mathcal{C}$
\item[(ii)] $P,Q$ lie on a conic $\mathcal{R}$ confocal to $\mathcal{X}$.
\item[(iii)] $S,T$ lie on a conic $\mathcal{G}$ confocal to $\mathcal{X}$.
\item[(iv)] $U,V$ lie on a conic $\mathcal{B}$ confocal to $\mathcal{X}$.
\end{itemize}
\end{theorem}

 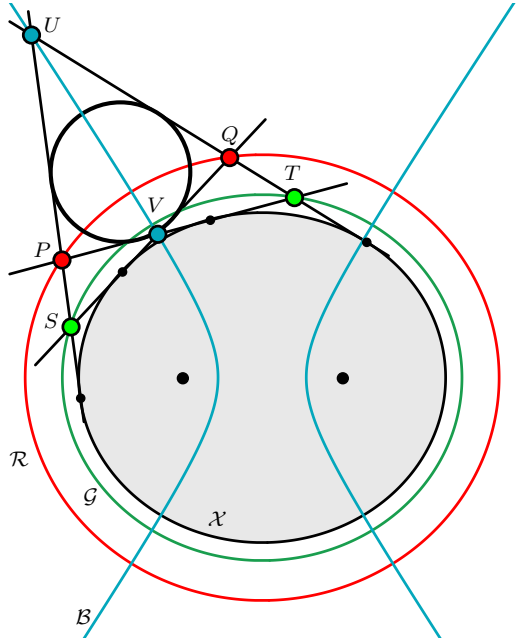
\begin{figure}[ht]
\centering
\includegraphics[width=.65\textwidth]{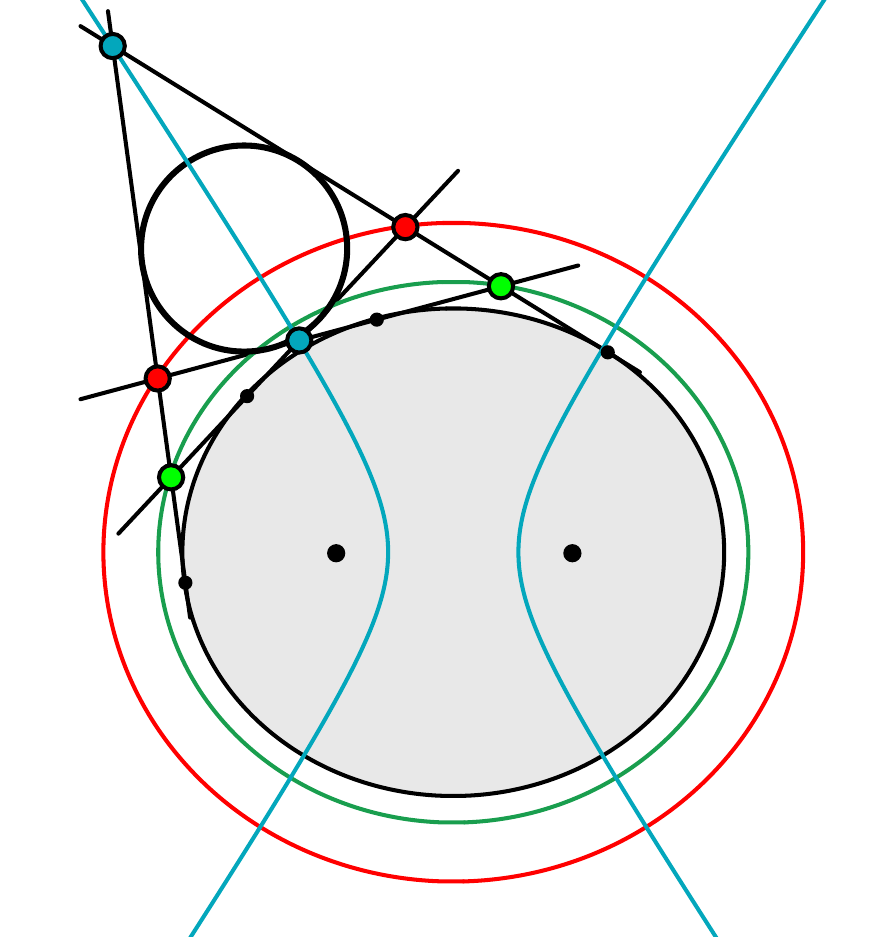}
\begin{picture}(0,0)
\put(-198,144){\footnotesize {$P$}}
\put(-127,188){\footnotesize {$Q$}}
\put(-194,117){\footnotesize {$S$}}
\put(-103,173){\footnotesize {$T$}}
\put(-194,229){\footnotesize {$U$}}
\put(-155,161){\footnotesize {$V$}}
\put(-132,42){\footnotesize {$\mathcal{X}$}}
\put(-207,65){\footnotesize {$\mathcal{R}$}}
\put(-179,51){\footnotesize {$\mathcal{G}$}}
\put(-182,5){\footnotesize {$\mathcal{B}$}}
\end{picture}
\caption{The geometric situation in Theorem~\ref{CGTfour}.}	
\label{fig:Thm3Conics}
\end{figure}

This version (see Figure~\ref{fig:Thm3Conics}) is a direct consequence of Theorem~\ref{CGTpure} and follows from the interchangeability of the three pairs of points in the quadrilateral $a,b,c,d$ that are supposed to lie on the conic $\mathcal{R}$ in Theorem~\ref{CGTpure}.  Furthermore, Theorem~\ref{CGTpure} is a consequence of Theorem~\ref{CGTtangent}. Hence, we will aim for carefully proving  Theorem~\ref{CGTtangent}.

For completeness, we also provide a flawed version in the spirit of Theorem~\ref{CGTfour}. We aim for a statement that uses the conics  $\mathcal{R}$ and  $\mathcal{G}$ (parts (ii) and (iii) above)
and in the `moreover'-part claims that the tangents related to them meet.
Here is a statement that contains the same logical error as we have encountered before:

\begin{notheorem}
{\bf (A  flawed version of CQT):} Let $\mathcal{X}$ be an ellipse in the Euclidean plane and let $a,b,c,d$ be four distinct tangents to $\mathcal{X}$. Assume that the points $P=a\meet b$ and $Q=c\meet d$ lie on a  conic $\mathcal{R}$ confocal to $\mathcal{X}$. Then the points $S=a\meet c$ and $T=b\meet d$ lie on a conic $\mathcal{G}$  confocal to 
the other two \textcolor{red}{\ul{and all four tangents that can be formed at $P,Q,S,T$ to the corresponding conics meet in a point}.}
\end{notheorem}
\noindent

The first part of this statement is true (it is the equivalence of (ii) and (iii) in Theorem  \ref{CGTfour}). However, the  red underlined statement about the tangents is only true in the generic case, while there are counterexamples in special situations.
Again, the pitfall comes from the symmetric situation in which the choice of the conics becomes ambiguous.
Figure~\ref{fig:tangents} shows two situations. The picture on the left shows the situation that corresponds to the generic case. The picture on the right shows the symmetric situation in which the choice of conic through $P$ and $Q$ becomes ambiguous, which  leads to the counterexample.
\begin{figure}[ht]
\centering
\includegraphics[width=.45\textwidth]{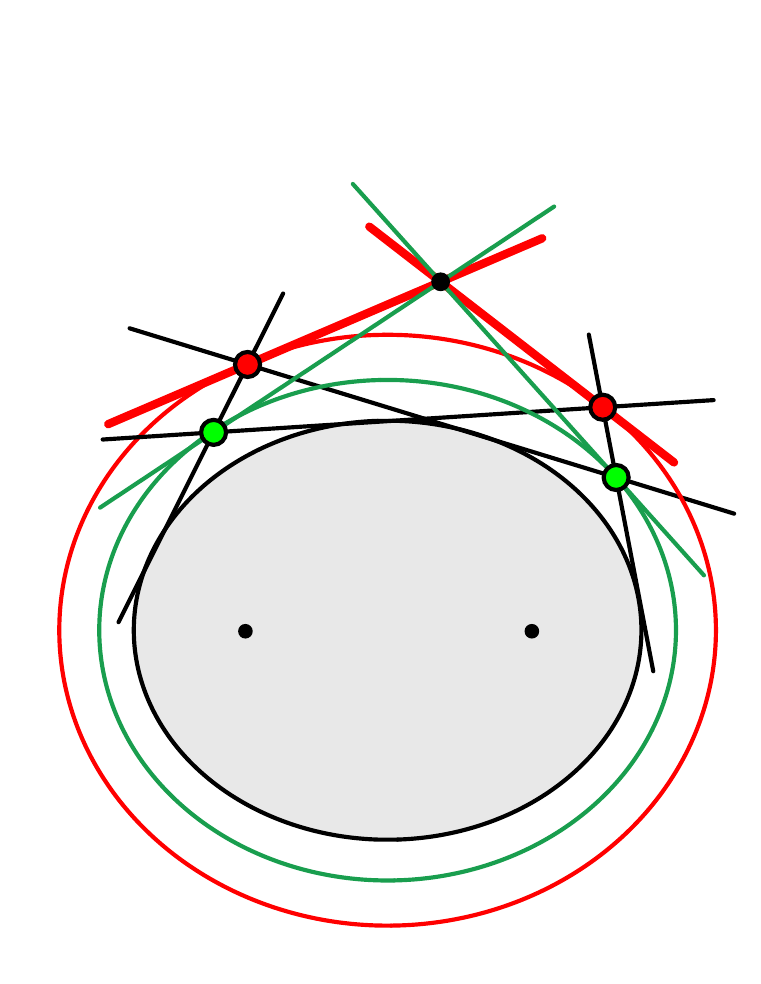}\;\; \quad
\includegraphics[width=.47\textwidth]{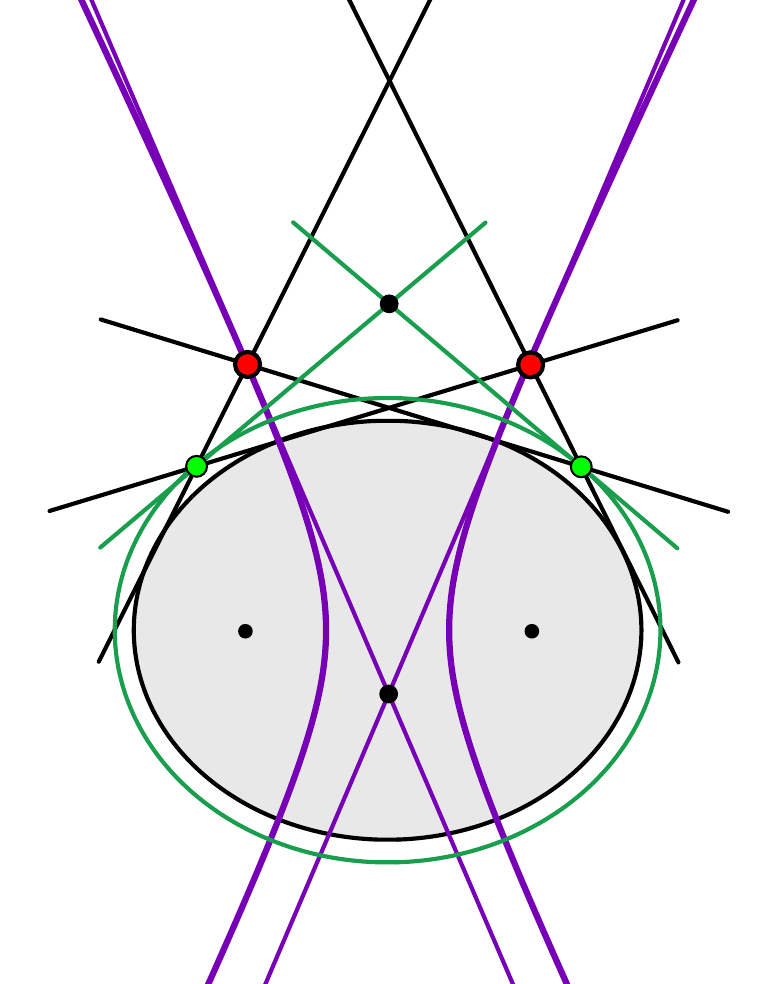}
\begin{picture}(0,0)
\put(-252,148){\footnotesize {$M$}}
\put(-295,132){\footnotesize {$P$}}
\put(-215,124){\footnotesize {$Q$}}
\put(-300,116){\footnotesize {$S$}}
\put(-212,104){\footnotesize {$T$}}
\put(-124,124){\footnotesize {$P$}}
\put(-48,126){\footnotesize {$Q$}}
\put(-130,113){\footnotesize {$S$}}
\put(-42,114){\footnotesize {$T$}}
\end{picture}
\caption{An instance of the configuration in Not–a–Theorem 1 in which the tangents meet (left) in contrast to an instance where they do not (right).}	
\label{fig:tangents}
\end{figure}
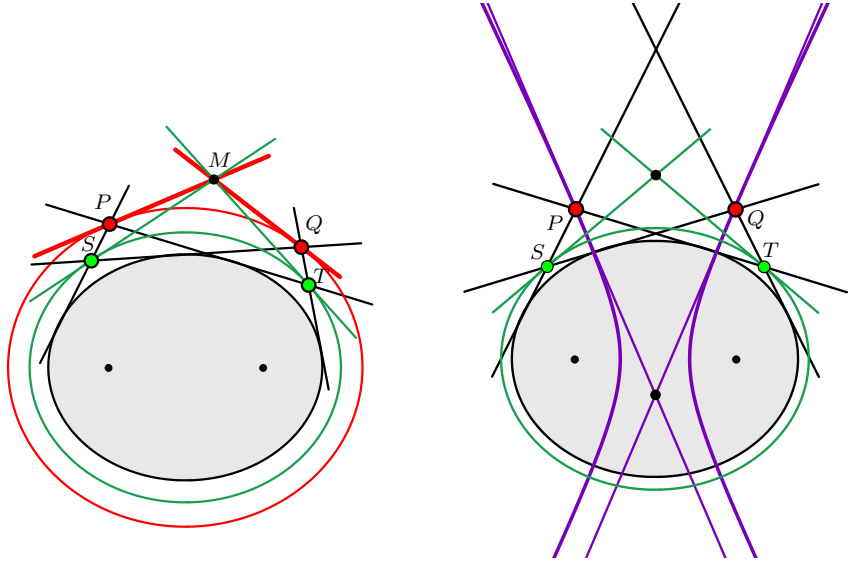

\begin{remark}
Indeed, a statement in the spirit of the last non-Theorem was exactly what we needed for our project related to Poncelet's Theorem and $(n_4)$ configurations \cite{BGRGT24a}. We needed a situation where we could ensure that four specific tangents to codependent conics meet in a point. \end{remark}

There are two ways to fix the statement. 
One is to specify conditions that force the selection of compatible conics; this is what we did in the appendix of \cite{BGRGT24a}, as indicated in the previous remark.
It turns out that in the particular case, where $\mathcal{X}$,
$\mathcal{R}$, $\mathcal{G}$ come from a Poncelet grid, the four tangents will indeed(!) always meet in a point.
The other way is to reconsider the construction order, like we did in Theorem~\ref{CGTtangent}.  Here is a version  that deliberately avoids the 
ambiguity by constructing the intersection of the four tangents in a early stage of the process. For the labelling again refer to Figure~\ref{fig:tangents} (left). A proof will be provided in Section \ref{sec:proofThm4}.

\begin{theorem} {\bf (CQT three conic version with tangents)}:\label{CGTtangent2}
Let $\mathcal{X}$ be an ellipse in the Euclidean plane and let 
$\mathcal{R}$,  $\mathcal{G}$, $P$ and $S$ be such that  $\mathcal{R}$ and  $\mathcal{G}$ 
are confocal to  $\mathcal{X}$, $P$ is on $\mathcal{R}$, $S$ is on $\mathcal{G}$, and the line $PS$ is tangent to $\mathcal{X}$. 
Let $M$ be the intersection of the tangents of $P$ to $\mathcal{R}$ and $S$ to $\mathcal{G}$.
 Let $Q$ and $T$ be the touching points of the other tangents from $M$ to $\mathcal{R}$ and $\mathcal{G}$, respectively. Then the lines $QT$, $PR$ and $SQ$ are all tangent to $\mathcal{X}$.
\end{theorem}
\smallskip


\section{The projective nature of the theorem}
Although the CQT makes statements about confocal conics and circles, it is intrinsically projective in nature.
Actually, it unfolds its entire beauty and symmetry if it is interpreted projectively in an appropriate way.
For this, we recall a highly appropriate way to interpret circles, midpoints, and foci of a conic in a purely projective setup (see \cite{RG11}, Chapter 18 and 19 for an extensive introduction).
\medskip 

\subsection{The absolute circle points}

Over the complex projective plane, all circles have two points in common. These are sometimes
called the {\it absolute circle points} and have the following concrete homogeneous coordinates
\[
\mathsf{I}=(1,i,0), \qquad \mathsf{J}=(1,-i,0).
\]
In homogeneous coordinates, a circle equation takes the general form
\[
x^2+y^2+a\cdot xz+b\cdot yz+ d\cdot z^2=0.
\]
A simple calculation shows that that plugging in the coordinates of $\mathsf{I}$ or $\mathsf{J}$ satisfies this equation, since
\[
1^2+(\pm i)^2+a\cdot x\cdot 0 +b\cdot y \cdot 0+ d\cdot 0^2=0.
\]
This is in nice accordance with the fact that any two conics in suitably generic position have four intersections.
In the case of two distinct circles, we can never experience more than two real intersections, since two others 
($\mathsf{I}$ and $\mathsf{J}$)  are already complex. The two points $\mathsf{I}$ and $\mathsf{J}$ span the line at infinity  $l_\infty$,  and taken as a pair, they may be considered as a degenerate dual conic, whose primal part is the line at infinity (considered as a double line).  Conversely, we can characterise a circle by the property that it is a conic which passes through $\mathsf{I}$ and $\mathsf{J}$.
\smallskip

The two absolute points can be used to construct (or calculate) the center of a circle in the following way.
Given a circle $\mathcal{C}$, we know that $\mathsf{I}$ and $\mathsf{J}$ are on it. The join of $\mathsf{I}$ and $\mathsf{J}$ is the line at infinity, and hence we can construct the polar of $l_\infty$ by simply intersecting the tangents to 
$\mathcal{C}$ at $\mathsf{I}$ and $\mathsf{J}$. For any circle, the polar of  $l_\infty$ is its center. Thus, the center of $\mathcal{C}$ is the intersection of the tangents to $\mathcal{C}$ at $\mathsf{I}$ and $\mathsf{J}$.
%
This fact has a nice generalization that allows the characterisation of the foci of any conic in purely projective terms using $\mathsf{I}$ and $\mathsf{J}$. Consider a general conic $\mathcal{X}$ that is neither degenerate nor a circle; then the points $\mathsf{I}$ and $\mathsf{J}$ do not lie on it. Thus, we can consider the tangents through $\mathsf{I}$ and $\mathsf{J}$ to the conic~$\mathcal{X}$. They are four distinct lines and they have six intersections (that naturally group into pairs). One of these pairs is $\{\mathsf{I}, \mathsf{J}\}$ itself. Let the other pairs be $\{{f_1}, {f_2}\}$ and
$\{{g_1}, {g_2}\}$;  these two sets are pairs of foci of $\mathcal{X}$. One might wonder
why a conic has {\it four} foci and not two. 
It turns out that for a real conic~$\mathcal{X}$, one of the pairs always consists of real points, while the other one always consists of complex points; foci that we just don't see. A proof of the correctness of this construction that ultimately goes back to a degenerate version of Pascal's Theorem can be found in \cite{RG11}, {Section 19.4.}
\begin{figure}[ht]
\centering
\includegraphics[width=.35\textwidth]{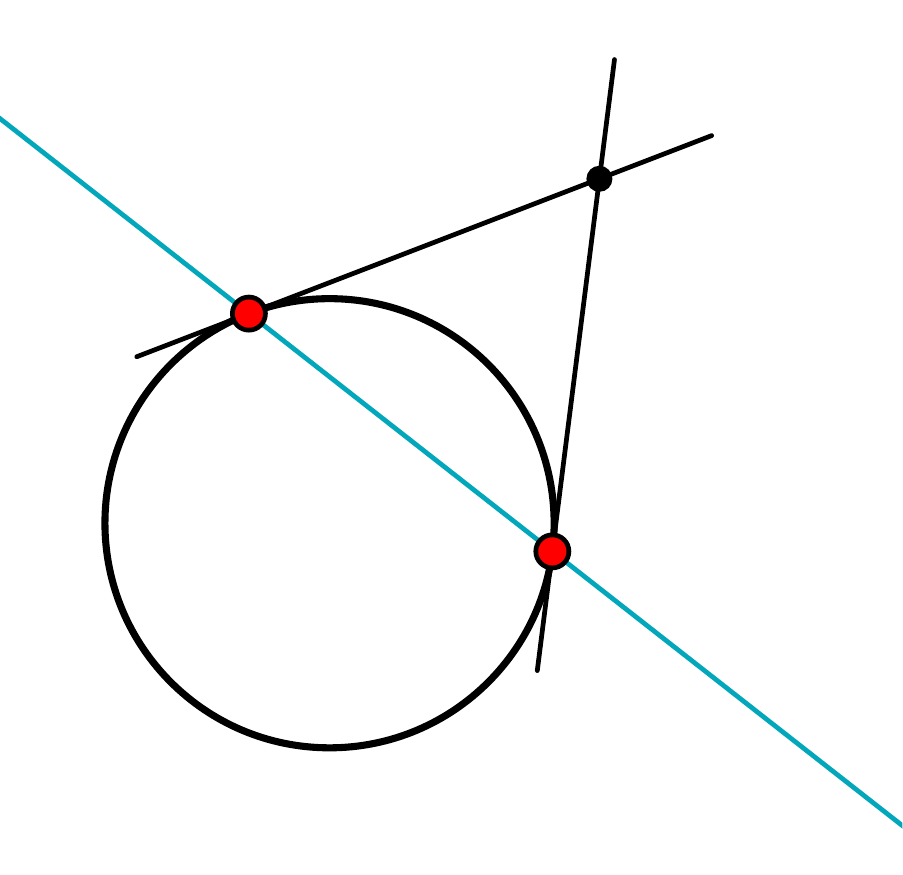}\qquad
\includegraphics[width=.55\textwidth]{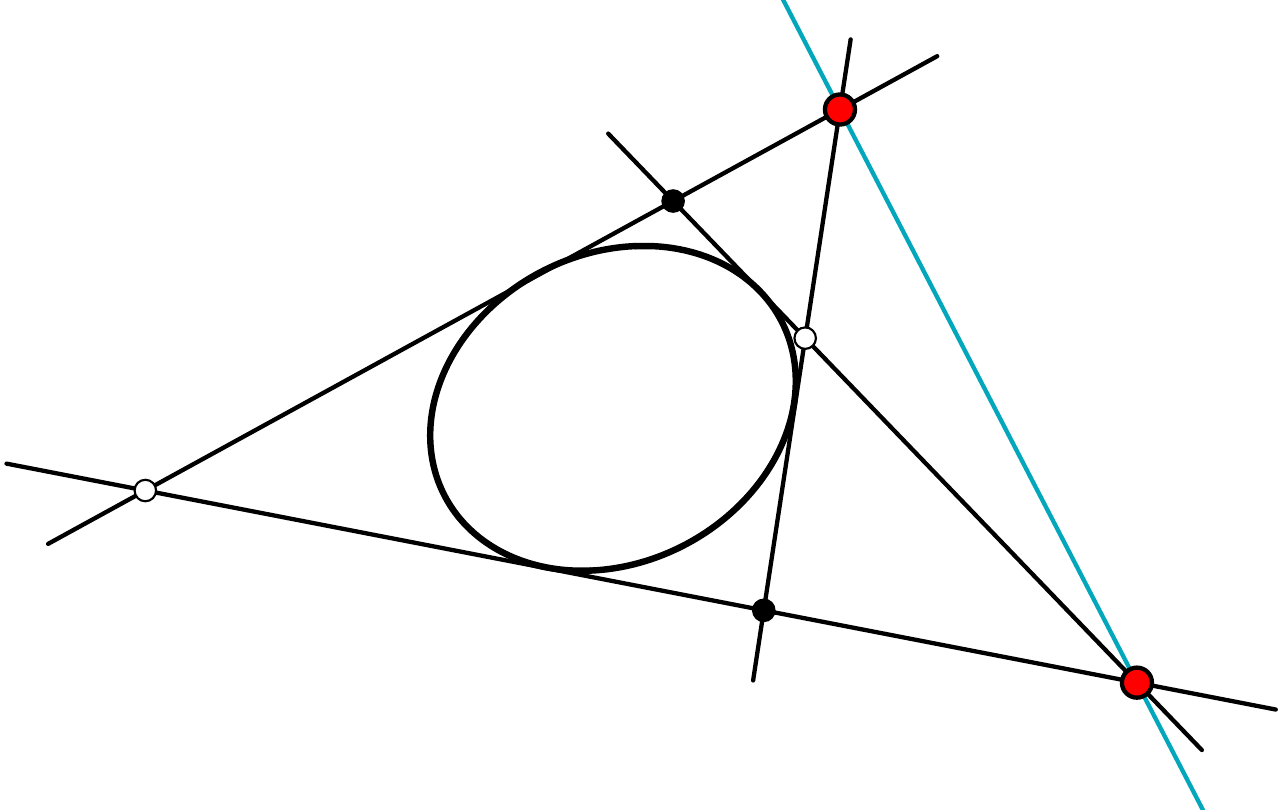}
\begin{picture}(0,0)
\put(-300,79){\footnotesize {$\mathsf{I}$}}
\put(-256,43){\footnotesize {$\mathsf{J}$}}
\put(-252,87){\footnotesize {$M$}}
\put(-232,7){\footnotesize {$l_\infty$}}
\put(-63,99){\footnotesize {$\mathsf{I}$}}
\put(-23,24){\footnotesize {$\mathsf{J}$}}
\put(-252,87){\footnotesize {$M$}}
\put(-30,2){\footnotesize {$l_\infty$}}
\put(-70,69){\footnotesize {$f_1$}}
\put(-175,35){\footnotesize {$f_2$}}
\put(-95,97){\footnotesize {$g_1$}}
\put(-92,23){\footnotesize {$g_2$}}
\end{picture}
\caption{Constructing the center of a circle and the foci of a conic.
{In this drawing the circular points $\mathsf{I}$ and $\mathsf{J}$ are transformed to real positions, to make all construction elements become real. The situation shown is projectively equivalent to the situation in which $\mathsf{I}=(1,i,0)$ and $\mathsf{J}=(1,-i,0)$.}
}	
\label{fig:foci}
\end{figure}
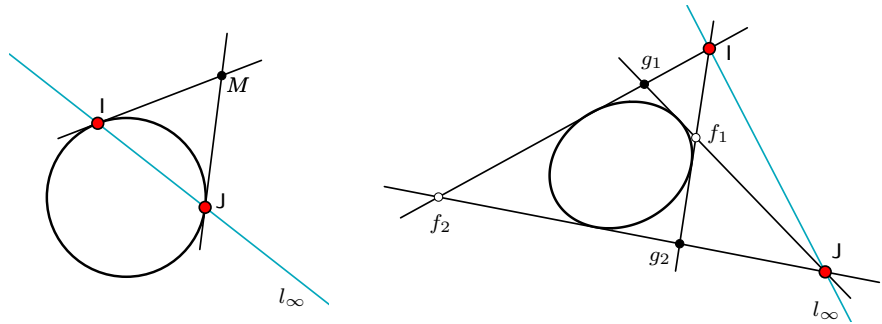

Figure~\ref{fig:foci} illustrates the two constructions by applying a projective transformation that moves the points  $\{\mathsf{I}, \mathsf{J}\}$ into a real and finite position 
The image highlights all essential incidence information about the construction.
\smallskip

In the light of this projective characterisation of foci, {\it conics being  confocal} has a nice purely incidence-theoretic characterisation: two conics share the same foci if their tangents to  $\{\mathsf{I}, \mathsf{J}\}$ are identical. 
In other words, a collection of confocal conics shares the same four tangents (two of its intersections are  $\mathsf{I}$ and $\mathsf{J})$. In even other words: they are {\it co-dependent}. If we represent a conic by a quadratic form $p^TAp=0$ where $A$ is a symmetric $3\times 3$ matrix, then three conics share the same four points if their corresponding matrices are linearly dependent. Dually, three conics are co-dependent if the \emph{inverses} of their matrices
are linearly dependent. In this case, they have four tangents in common.

\medskip
\subsection{A projective view of the CQT}
By introducing the special circular points  
$\mathsf{I}$ and $\mathsf{J}$, we reduced Euclidean concepts like {\it circle, midpoint, foci} to statements that can be expressed entirely in projective terms.
In these descriptions, everything is encoded in  terms of {\it points, lines, conics, incidences} and {\it tangencies}. Every such statement and construction remains invariant under projective transformation. This allows us to derive a purely projective description of Euclidean statements, which is what we implicitly did in the schematic picture in Figure~\ref{fig:foci}.

This allows us to translate the CQT into a purely projective description. We express \emph{confocality} by ``sharing tangents, two of whose intersections are  $\mathsf{I}$ and $\mathsf{J}$'', \emph{circle} by ``conic through  $\mathsf{I}$ and $\mathsf{J}$'' and \emph{center of the circle} by ``the point where the tangents at  $\mathsf{I}$ and $\mathsf{J}$ intersect''. Since this way to express the statement is entirely based on  projective concepts, we do not have to require that  $\mathsf{I}$ and $\mathsf{J}$ take any special position.  That is, after a projective transformation of the situation, $\mathsf{I}$ and $\mathsf{J}$ can be two arbitrary points rather than the special complex points defined above, and the statement still holds true and is equivalent to the original statement.

It is difficult to make a drawing of this projective situation that does not look overcrowded and unintelligible. An attempt is made in Figure~\ref{fig:CGTProj}. The following theorem is a projective translation of Theorem~\ref{CGTtangent}. For better readability we mention the corresponding color of the objects in the theorem. 

\noindent
\begin{theorem} {\bf (CQT with tangents, projective version)}:\label{CGTproj}
Let $t_1\upto t_4$ be four lines (gray) in general position in the (real or complex) projective plane.
Let $\mathsf{I}$ and $\mathsf{J}$ (green points) be one pair of intersections.
Let  $\mathcal{X}$ (black) and  $\mathcal{R}$ (red) be two distinct conics tangent to $t_1\upto t_4$.
Let $P$ and $Q$ (red) be two distinct points on  $\mathcal{R}$. Consider the four tangents $a,b,c,d$ (blue lines) through $P$ and $Q$ to
the conic $\mathcal{X}$. Then there is a conic  $\mathcal{C}$ (green) tangent to $a,b,c,d$ passing through  $\mathsf{I}$ and $\mathsf{J}$ and such that the tangents at $\mathsf{I}$ and $\mathsf{J}$ (green lines) to
$\mathcal{C}$ and the tangents at $P$ and $Q$ (red lines) to $\mathcal{R}$ meet in a point.
\end{theorem}

\begin{figure}[ht]
\centering
\includegraphics[width=.75\textwidth]{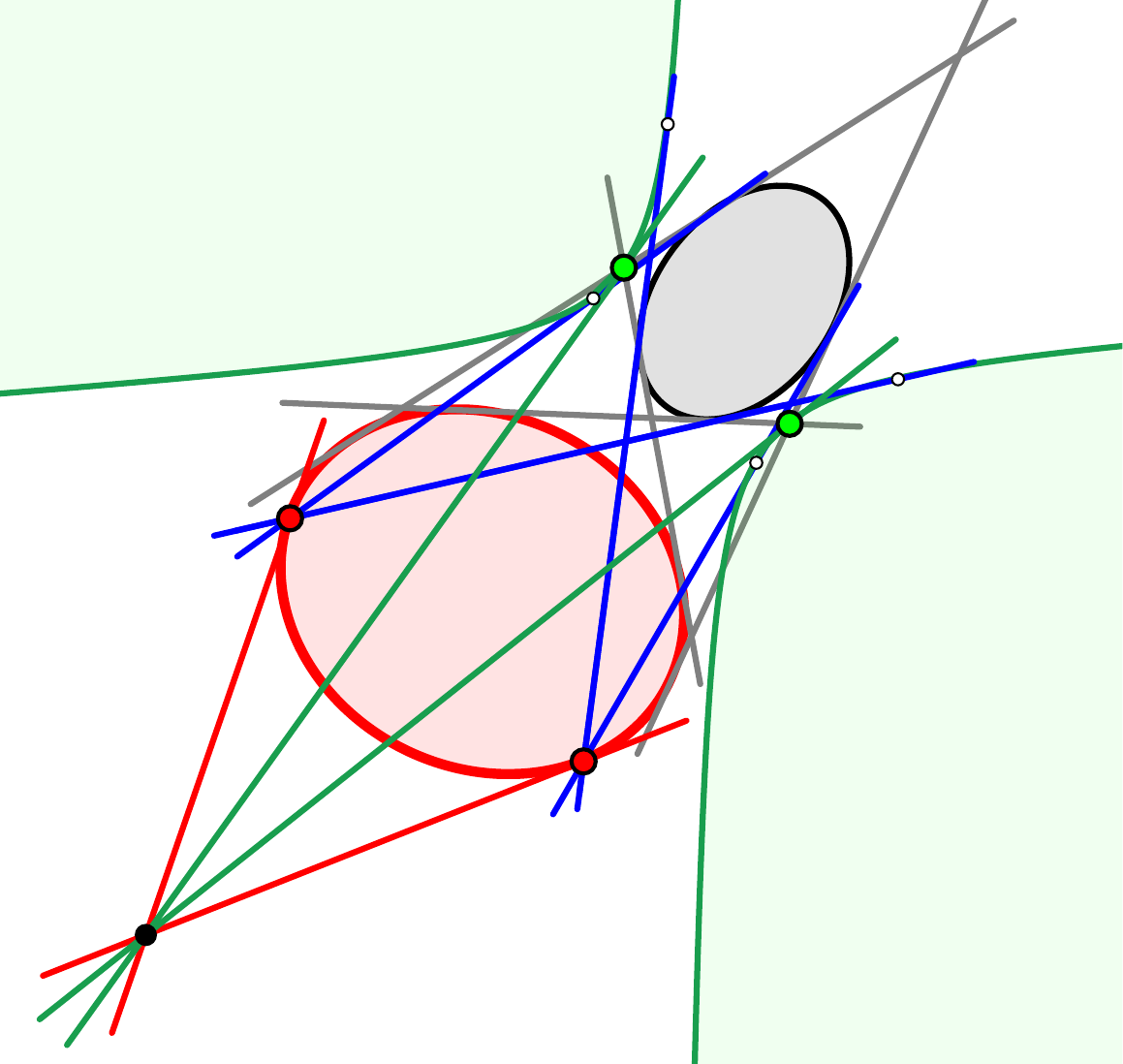}
\begin{picture}(0,0)
\put(-210,134){\footnotesize {$t_1$}}
\put(-205,152){\footnotesize {$t_2$}}
\put(-125,212){\footnotesize {$t_3$}}
\put(-37,248){\footnotesize {$t_4$}}
\put(-94,174){\footnotesize {$\mathcal{X}$}}
\put(-154,110){\footnotesize {$\mathcal{R}$}}
\put(-54,90){\footnotesize {$\mathcal{C}$}}
\put(-125,185){\footnotesize {$\mathsf{I}$}}
\put(-80,137){\footnotesize {$\mathsf{J}$}}
\put(-220,124){\footnotesize {$a$}}
\put(-212,110){\footnotesize {$b$}}
\put(-140,50){\footnotesize {$c$}}
\put(-130,52){\footnotesize {$d$}}
\put(-	192,118){\footnotesize {$P$}}
\put(-	137,77){\footnotesize {$Q$}}


\end{picture}
\caption{A projective interpretation of the CQT.}	
\label{fig:CGTProj}
\end{figure}
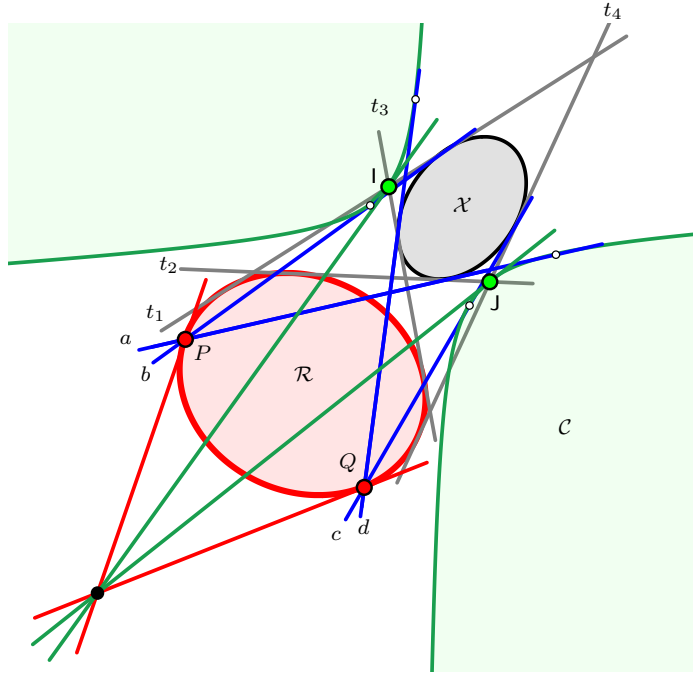

We  did not include the second part of Theorem~\ref{CGTtangent} in this projective translation because it is logically the same statement. It arises when one systematically interchanges the roles
of 
$\mathcal{R} \leftrightarrow \mathcal{C}$, 
$\{P,Q\} \leftrightarrow \{\mathsf{I},\mathsf{J}\}$, ``gray"  $\leftrightarrow$ ``blue" and
``red"  $\leftrightarrow$ ``green". Translating Theorem~\ref{CGTtangent} into the projective world unveils the inherent underlying
symmetry of the statement. 

\medskip
Notice that there are two types (gray and blue) of {\it four tangents} involved in this situation. To get a picture of the configuration that is a bit more accessible to the human mind we projectively dualize the situation. We will call the situation in Figure~\ref{fig:CGTProj} the {\it dual} situation. Switching to the {\it primal} situation
by dualizing, points become lines and vice versa. \emph{Co-dependent} conics that are tangent to 
four lines become \emph{dependent} conics having four points in common.
Figure \ref{fig:CGTProjDual} shows the dualized situation, which is optically significantly less crowded.
We may formulate the primal version as follows:
\begin{theorem} {\bf (CQT with tangents,  primal version)}: \label{CGT5} 
Let $\mathcal{X}$ and $\mathcal{R}$ be two conics that have four distinct (gray) points in common.
Consider two lines $\mathsf{I}$ and $\mathsf{J}$ through two pairs of these points.
Consider two tangents at $P$ and $Q$ to $\mathcal{R}$ and intersect them with $\mathcal{X}$ in four blue points. Then there is a conic $\mathcal{C}$ through the blue points that is simultaneously tangent to 
$\mathsf{I}$ and $\mathsf{J}$  and such that the touching points at $\mathsf{I}$ and $\mathsf{J}$ and
 $P$ and $Q$ are collinear.
\end{theorem}
\noindent
Again the symmetry of the statement is prominent and clearly visible in the corresponding picture. Proving this version of the  theorem automatically proves all other versions that we previously stated.
\medskip

\begin{figure}[ht]
\centering
\includegraphics[width=.75\textwidth]{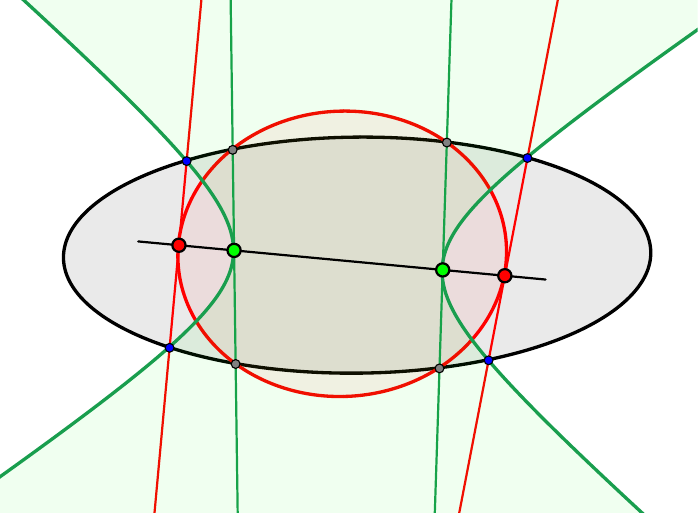}
\begin{picture}(0,0)
\put(-172,180){\footnotesize {$\mathsf{I}$}}
\put(-103,180){\footnotesize {$\mathsf{J}$}}
\put(-199,180){\footnotesize {$P$}}
\put(-53,180){\footnotesize {$Q$}}
\put(-238,68){\footnotesize {$\mathcal{X}$}}
\put(-244,168){\footnotesize {$\mathcal{C}$}}
\put(-158,36){\footnotesize {$\mathcal{R}$}}

\end{picture}
\caption{The primal version of the projective interpretation of the CQT.}	
\label{fig:CGTProjDual}
\end{figure}
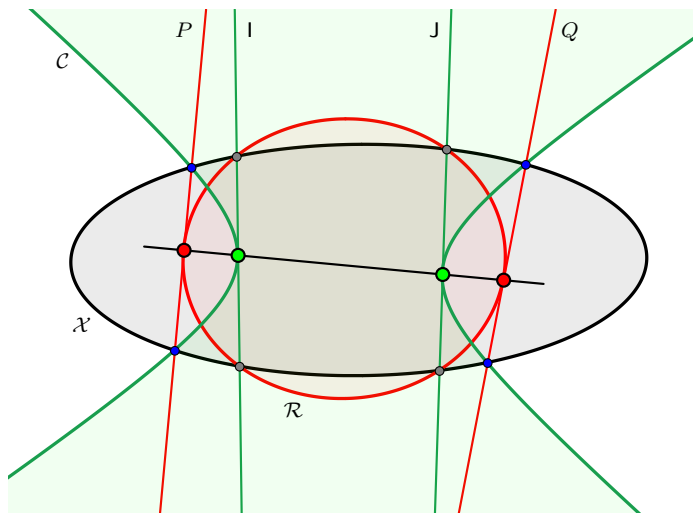

One might wonder where the ambiguity that caused the problems in the flawed  formulation of the original CQT is hiding. We explain the situation in the primal projective picture. After the black and red conic and 
the red and green lines are drawn, we ask for a conic that passes through the four blue points and is tangent to one of the green lines. However, the geometric operation ``conic through 4 points and tangent to 1 line"
is not uniquely defined. In general, there are two conics satisfying this property, since algebraically this leads to the problem of solving a quadratic equation. These two conics are the equivalent of the two possible choices of circles in, for instance, Darboux's original formulation. In a generic situation, only  one of those two conics will also be tangent to the other green line, but a special case might still arise where both conics are tangent to both green lines.
However, only one of them will have the property that the four touching points (the red and the green ones)
are at the same time collinear. In our carefully chosen formulation of the theorems we took care of this by asking for  \emph{both} requirements at once.
\medskip

In a sense, the conclusion of  the theorem states several things at once (tangency to the two green lines and collinearity of the four points). We can use these multiple conclusions to reorganise the construction sequence of the elements in the drawing so that we do not run into the ambiguous situation (in other words, we avoid asking for the construction of a conic that might introduce ambiguities). This can be done as follows.
Start with the black and the red conic, and then
draw the two red tangents to the red conic, the two green lines and the four blue points. Now, proceed as follows. 
\begin{itemize}
\item Construct the two red touching points of the red tangents.
\item Connect them by a black line.
\item Intersect this line with the two green lines to also obtain the green touching points.
\item Construct the {\it unique} (!) conic tangent to the green lines at the green intersection points that also passes through one of the gray points.
\end{itemize}

This conic  will now simultaneously satisfy three 
additional incidences: it will pass through the remaining three gray points (we will prove this  in the next section).
A little remark about the uniqueness of the construction of the green  conic is appropriate here. We ask for
a conic that passes through one point (gray) and at two places has a specified tangent (green lines) at a specified position at this tangent (green points). The fact that the green points lie on the green lines removes the ambiguity by forcing the (four) different solutions of a general ``conic-through-4-points-tangent-to-1-line''
operation to coincide.

\medskip
If we retranslate this version to the original case of the Euclidean CQT (version in Theorem~\ref{CGTtangent}) about confocal conics and lines, the black line plays the role of the circle center $M$ and the green conic plays the role of the circle. The statement then reads as follows (in the original dual world): 
If we take a circle (green conic through $\mathsf{I}$ and $\mathsf{J}$) with center $M$ (the polars of  $\mathsf{I}$ and $\mathsf{J}$ intersect in~$M$) that is tangent to one of the  lines $a,b,c,d$, then it is automatically tangent to the other three. This is exactly what the first part  the Theorem~\ref{CGTtangent} says. A similar translation can be done for the second part of Theorem~\ref{CGTtangent}  (which is in essence symmetric to the first one, as we saw). Observe that  this statement ultimately ends in three incidence statements that happen simultaneously: 
{\it if a circle with given center $M$  is tangent to a certain line, it will automatically be tangent to three others}.

This occurrence of three incidences happening at once strongly hints to a connection to the famous Cayley--Bacharach Theorem for which exactly such simultaneous conclusions arise. 
In fact, it turns out that our formulation is nothing else than a special case of the Cayley--Bacharach Theorem. We will explore this relation in the next section.

\medskip

\section{The Cayley--Bacharach Theorem approach}
So far, we have only rearranged the statement in various equivalent ways.
 We now turn to its actual proof. If we prove Theorem~\ref{CGT5}, then we automatically have proved Theorem~\ref{CGTtangent}, since it is just the dual of Theorem~\ref{CGT5} with  $\mathsf{I}$ and $\mathsf{J}$ moved to their actual positions. From there, all other theorems follow. Thus, our next goal will be proving Theorem~\ref{CGT5}.
\medskip

\subsection{Cayley--Bacharach arguments}

There is a surprising relationship between the CQT and the famous Cayley--Bacharach Theorem about 
incidence patterns on plane algebraic curves. Besides the fact that the simplest instance of the Cayley--Bacharach Theorem, which makes a statement about the incidence patterns of three cubics, was first discovered by Chasles, Chasles' original argument for 
the CQT used a predecessor of the Cayley--Bacharach Theorem. We will elaborate on this later.

Recall that by Bezout's Theorem, two plane algebraic curves 
of degree $d_1$ and degree $d_2$ that do not have infinitely many points in common intersect in $d_1 \cdot d_2$ points. For the proper count, 
it is important (as usual) to take  multiplicities of intersections, complex intersections and intersections at infinity into account. Thus, two curves of degree $d$ will have $d^2$ intersections. 

The first instance of what is now called the Cayley--Bacharach Theorem was spelled out by Chasles \cite{Cha1865}. It is a theorem about the intersection pattern of cubics in the projective plane:
\begin{theorem}\label{CBT3}
Let $\mathcal{A}$ and $\mathcal{B}$ be two curves of degree 3 intersecting in 9 distinct points. Any other degree 3 curve that passes through 8 of these intersections will automatically contain the ninth point.
\end{theorem}

The story of the generalisation of this statement to higher degrees (and higher dimensions) 
is one of the deep success stories of algebraic geometry and the origin of many important concepts.  It is somewhat involved, since it included 
several false attempts and needed lots of refinement over the decades. 
Several formulations were proposed, and many of them suffered from special situations that were not properly covered.
For an excellent overview of the 
development and further directions, see \cite{EGH96}.
Here, we  only need the version, stated as Theorem \ref{CBT4}, that is the next more-complicated case 
after Theorem~\ref{CBT3}, about the intersection pattern of 16 points formed by intersecting  two quartics.  The statement directly follows from the (significantly more general) theorem proved by Bacharach  by plugging in the appropriate degree parameters. The general degree $d$ version can be found in the original 1895 paper of Bacharach  \cite{Bach1886}.\footnote{{This very powerful theorem is called} {\it Cayley--Bacharach Theorem} {since a first version of the statement was published by Cayley in 1843, generalising a Theorem of Chasles who stated the most elementary case for three cubics. However, Cayley's formulation  was lacking a crucial non-degeneracy condition. A corrected statement was finally formulated and proved by Isaak Bacharach} \cite{Bach1886}.
}
 
\begin{theorem}\label{CBT4}
Let $\mathcal{A}$ and $\mathcal{B}$ be two plane algebraic curves of degree $4$
intersecting in $16$ distinct points $\{p_1\upto p_{16}\}$.  Furthermore, let $p_{14},p_{15},p_{16}$ be noncollinear. Then any curve $\mathcal{C}$ of degree $4$ that contains  all the points
 $p_1\upto p_{13}$ also contains $p_{14},p_{15},p_{16}$.
\end{theorem}

\medskip

\begin{figure}[ht]
\centering
\includegraphics[width=.38\textwidth]{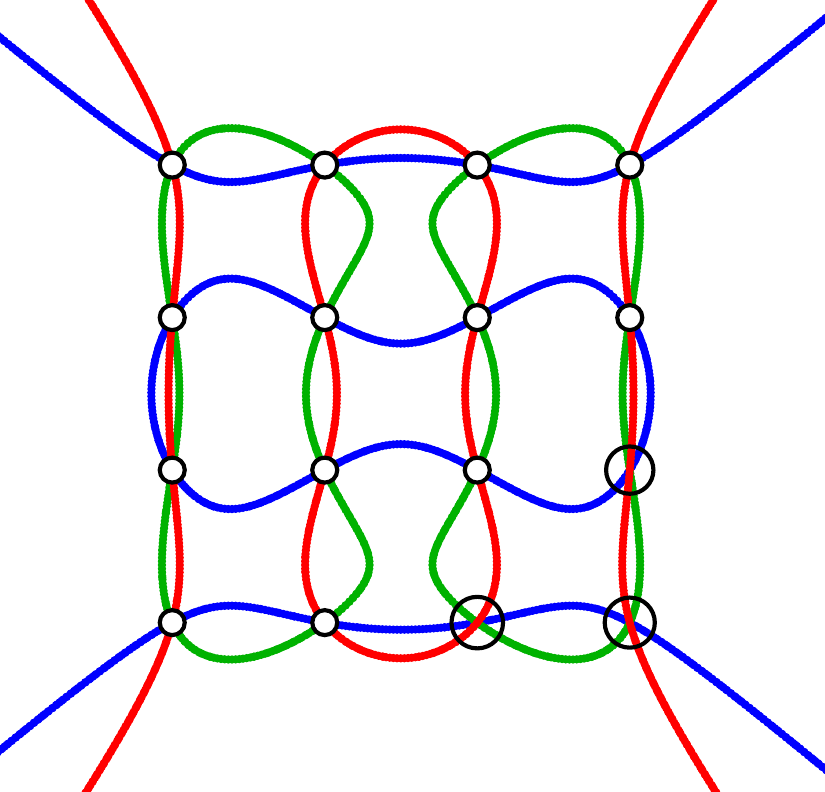}$\quad\;\;$
\includegraphics[width=.54\textwidth]{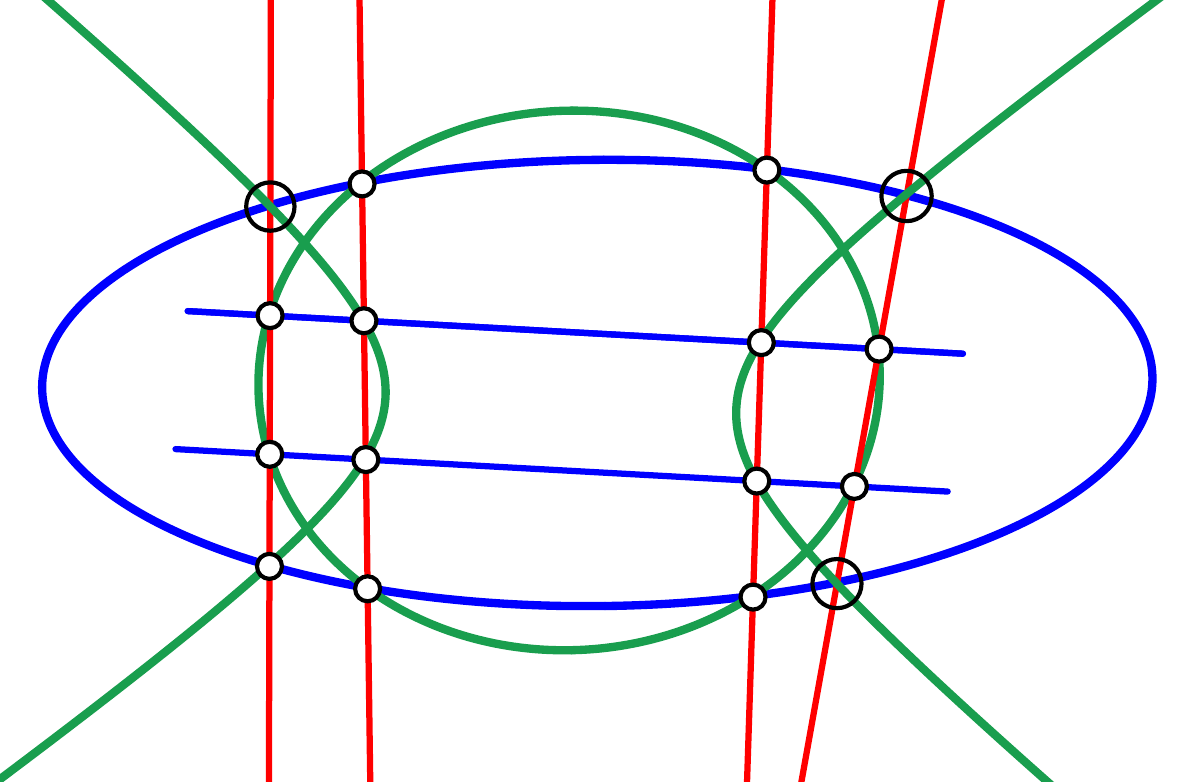}

\begin{picture}(0,0)
\end{picture}
\caption{The Cayley Bacharach Theorem for quartics in a  general (right) and a very special (left) situation. The additional incidences $p_{14}$, $p_{15}$, $p_{16}$ are indicated by small circles.}	
\label{fig:CBT1}
\end{figure}

Figure \ref{fig:CBT1} (left) illustrates the general situation. All curves shown in this picture (red, green and blue) are quartics. They have been chosen to pass through the 13 white points. The three additional incidences that are marked by small circles are the consequence of Theorem \ref{CBT4}.
We intentionally chose a situation that is quite regular so that it is easy to count the points, but the theorem holds for arbitrarily wild configurations.
\medskip

Our use case of this theorem will be a special case, shown in Figure \ref{fig:CBT1} (right), in which each quartic  decomposes into some combination of conics and lines, as shown in the image.
The short version of what we get is: {\it If the 13 coincidences marked by white points are present, the other three incidences follow automatically.} With a few more details:

\noindent
{\bf Special case of Theorem \ref{CBT4}:}
{\it
Assume the red quartic decomposes into four lines and the blue quartic decomposes into
two lines and a conic, generating 16 points of  intersection. For all involved conics and lines, assume the specific intersection combinatorics as indicated by the picture.
Consider a green conic passing through~8 of these points (the green ellipse)
and another green conic passing through~5 of the remaining points (the green hyperbola).
Then the three additional intersections between the red lines and the 
blue conics must also lie on the last green conic (marked by little circles). }

To see that this statement is just an instance of Theorem \ref{CBT4}, observe that the four red lines form a quartic and the two blue lines together with the blue conic form a quartic as well.
Furthermore, assume that the blue conic is not degenerate, and
assume that we have 16 distinct points of intersection. 
The three intersections marked by circles all lie on the non-degenerate blue conic, so they are not collinear. Hence, any quartic passing through the white points must also pass through those three points by Theorem \ref{CBT4}.
 The two green conics together form a quartic that passes through the 13 white points, and the three intersections marked by circles are still not collinear, so we can apply Theorem \ref{CBT4} and see that the green quartic must also be incident to the remaining three intersections. 
These three points cannot lie on the green ellipse since the red lines already have two intersections with that ellipse. Hence, they must lie on the other green conic (the hyperbola).

\medskip
\begin{figure}[ht]
\centering
\includegraphics[width=.47\textwidth]{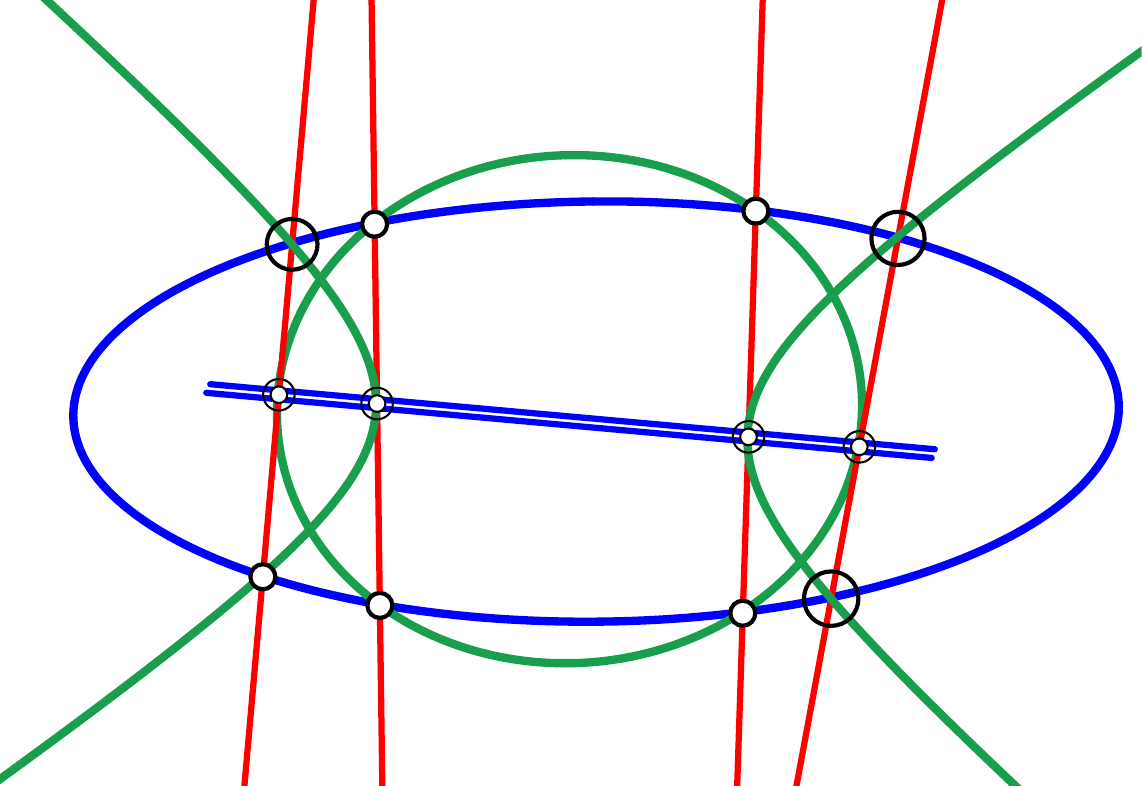}$\quad$
\includegraphics[width=.47\textwidth]{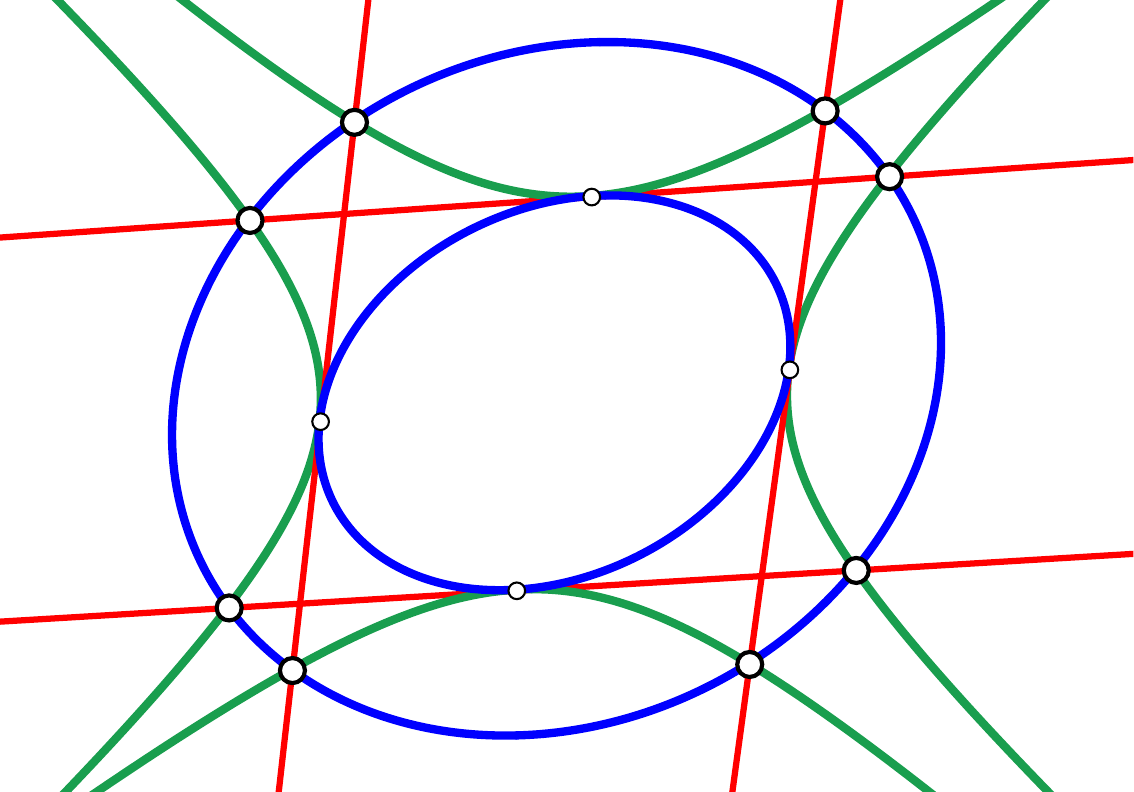}

\begin{picture}(0,0)
\end{picture}
\caption{The limit case in which the inner two lines become a double line (left) vs. an  analogous instance where the touching points are not collinear (right).}	
\label{fig:CBT2}
\end{figure}

\begin{proof}{\it (of Theorem~\ref{CGT5} and all the other theorems):}
The only thing that is missing in order to prove Theorem~\ref{CGT5} is to move 
the configuration we just considered into a slightly degenerate situation by making
the two blue lines coincide. For this, consider the situation of Figure \ref{fig:CBT2} on the left.
We argue geometrically (the argument can also made rigorous by performing the limit process analytically). We may consider the intersection of the blue lines $b_1$ and $b_2$ in Figure
\ref{fig:CBT1} on the right and consider the pencil of lines defined by $b_1+ \lambda b_2$.
If $\lambda$ approaches $0$ the points of intersection of each of the blue lines with the four red lines
will get closer and closer until they finally coincide at $\lambda=0$. During this process 
they will become touching points for the green conics. The final situation is shown in 
Figure \ref{fig:CBT2} on the left. The coincidence at the points marked by the circles will be present during the entire process as well as in the limiting case. Comparing this situation 
with Figure \ref{fig:CGTProjDual}, we see that this is exactly the same situation (just with a different color coding). Through our discussion in the last section, this proves Theorem~\ref{CGT5}.
\end{proof}

By  constructing the blue double line prior to  our conclusion we  forced the geometric situation to be unique. One might wonder where the ambiguity that originally led to the flawed formulations of the CQT hides in \emph{this} context. The crucial point lies in the ambiguity that arises when we construct a conic tangent to one line and incident to four points.
The corresponding situation is shown in Figure \ref{fig:CBT2} on the right.
It arises from the same incidence pattern, but a reversed construction sequence.
We start with the outer blue conic and the four red lines, and we form the eight intersections of red and blue lines. Next, we ask for the green conics through four white points
and tangent to the red lines (combinatorics as indicated by the picture). There {\it may} be two ways to choose each of the green conics. The Cayley--Bacharach Theorem tells us that the four touching points must belong to a second blue conic and that these  points must be double points of the intersection of this conic with the red lines. However, this might be satisfied by two different intersection combinatorics, corresponding to the left and right of Figure \ref{fig:CBT2}.

\subsection{Historical remarks}

Some additional historical remarks may be appropriate here. The original argument of Chasles followed a related, but slightly different way of reasoning. He considered the configuration  arising  as a compound of 6 conics (two lines taken together forming a degenerate conic, and double lines also taken to be a degenerate conic). The equations of three conics are linearly dependent if the conics have four points in common (again, counted with multiplicities).
In our picture, the 16 points of intersection can be subdivided into 4 disjoint quadruples,
each of them expressing a dependency between 3 conics. Figure \ref{fig:CBT3} illustrates this situation. It is exactly the same image we already had, but now the objects are drawn with thin and thick lines to differentiate the six conics (there is a thin and a thick conic for each color).
We now consider the equations of  the conics as \emph{points} in the 5-dimensional projective space  $\mathbb{P}^5$ of all planar conics. Being dependent means that  the corresponding conic-points in that 5-dimensional space are collinear. The way Chasles thought about this situation was as follows:
The fact that $\mathcal{B}_2$, $\mathcal{R}_1$ and $\mathcal{G}_2$ have four points in common
means that the three conic-points are collinear in $\mathbb{P}^5$. Similarly, 
the fact that $\mathcal{B}_2$, $\mathcal{R}_2$ and $\mathcal{G}_1$ have four points in common translates into collinearity of the corresponding conic-points in $\mathbb{P}^5$. Since the two lines in $\mathbb{P}^5$ intersect in the conic-point $\mathcal{B}_2$, they span just a 2-dimensional plane in $\mathbb{P}^5$. Hence, the lines spanned by $\mathcal{R}_2$ and $\mathcal{G}_2$ and by $\mathcal{R}_1$ and $\mathcal{G}_1$ must intersect in a (conic-)point in $\mathbb{P}^5$. This point is precisely the conic corresponding to $\mathcal{B}_1$, and it must, therefore, have four points in common with $\mathcal{R}_2$ and $\mathcal{G}_2$ and four other points in common with
$\mathcal{R}_2$ and $\mathcal{G}_2$, which is essentially the statement of the claim.
\medskip

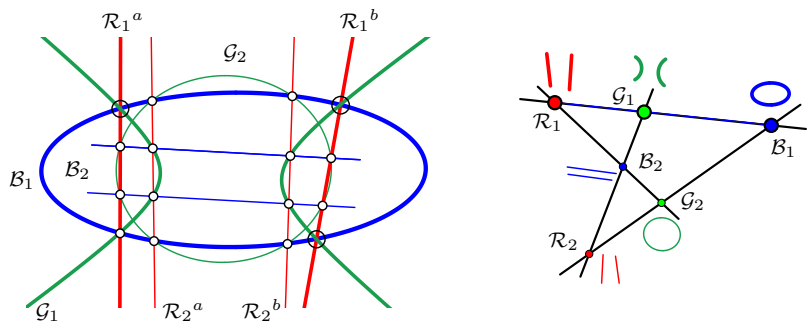
\begin{figure}[ht]
\centering

\includegraphics[width=.45\textwidth]{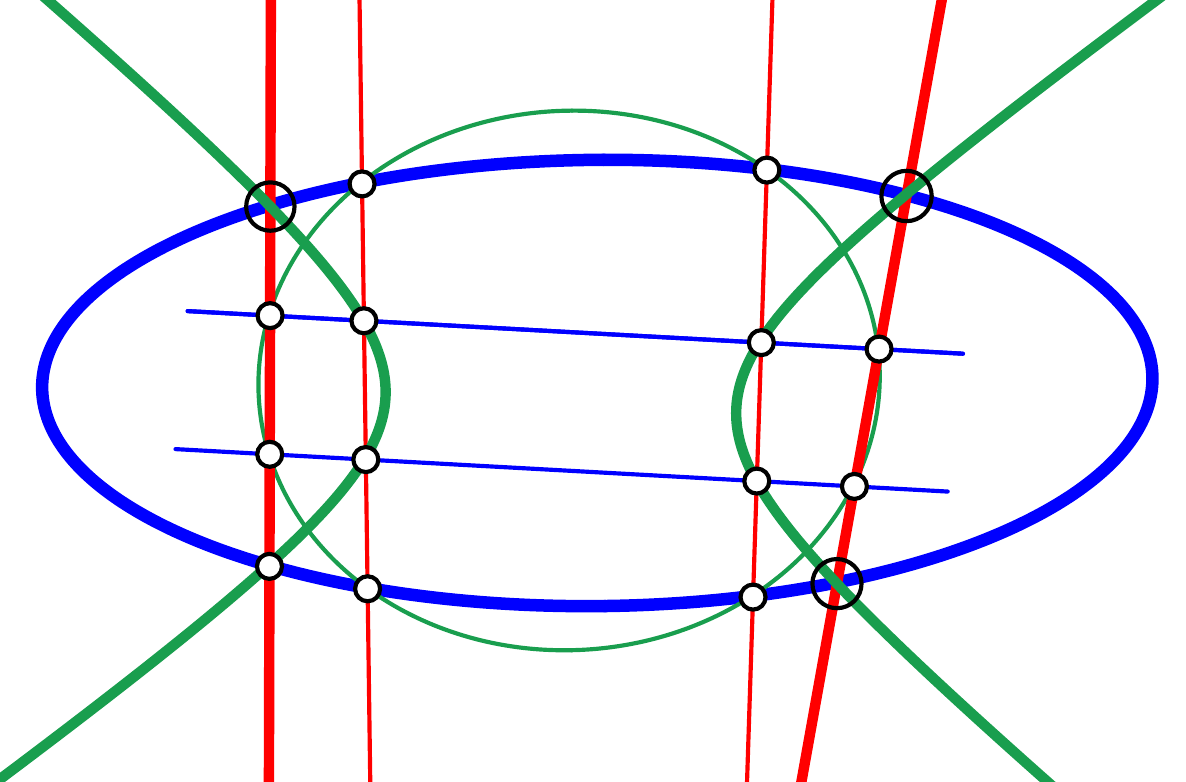}$\qquad$
\includegraphics[width=.36\textwidth]{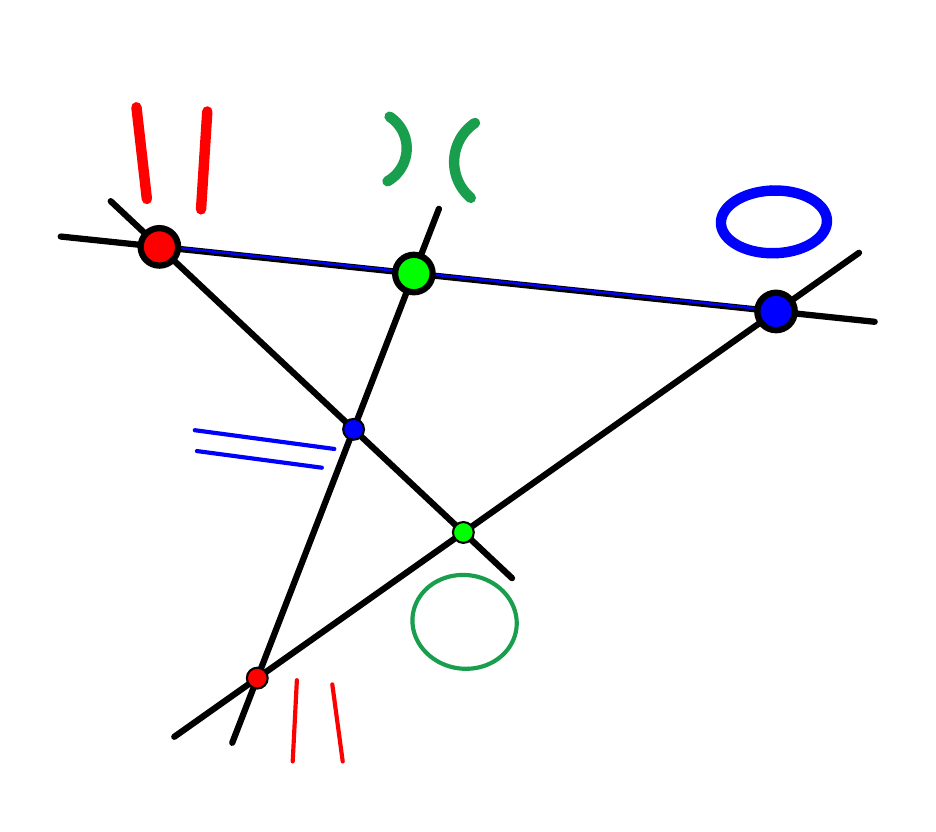}

\begin{picture}(0,0)
\put(-148,8){\footnotesize {$\mathcal{G}_1$}}
\put(-158,58){\footnotesize {$\mathcal{B}_1$}}
\put(-137,62){\footnotesize {$\mathcal{B}_2$}}
\put(-77,106){\footnotesize {$\mathcal{G}_2$}}
\put(-123,118){\footnotesize {${\mathcal{R}_1}^a$}}
\put(-33,118){\footnotesize {${\mathcal{R}_1}^b$}}
\put(-100,8){\footnotesize {${\mathcal{R}_2}^a$}}
\put(-70,8){\footnotesize {${\mathcal{R}_2}^b$}}
\put(40,80){\footnotesize {${\mathcal{R}_1}$}}
\put(46,35){\footnotesize {${\mathcal{R}_2}$}}
\put(80,65){\footnotesize {${\mathcal{B}_2}$}}
\put(70,90){\footnotesize {${\mathcal{G}_1}$}}
\put(96,50){\footnotesize {${\mathcal{G}_2}$}}
\put(130,70){\footnotesize {${\mathcal{B}_1}$}}

\end{picture}
\caption{Interpreting the CQT as a Chasles' quadrilateral. (left) The situation of CQT, as intersections among six conics. (right) The conics represented as points with collinearities in the projective space of conics.}	
\label{fig:CBT3}
\end{figure}
\medskip

In fact, Chasles did much more. He stated four meta-theorems of this kind for general conics, and from them he derived numerous specific claims by specialising them into various degenerate situations of all kinds. In particular, in the dual space (which was the setup in which Chasles originally represented his thoughts) one gets many surprising facts about confocal conics and circles (like our Chasles' Quadrilateral Theorem). Unfortunately, in stating these Corollaries of his meta-theorems, he created several flawed statements that suffer from the ambiguous situations. 
\medskip

Nevertheless, Chasles was way ahead of his time. His 1843 article \cite{Cha1843}  must have appeared close to metaphysics to his contemporaries. He wrote a second article clarifying many of his thoughts
in 1860 \cite{Cha1860}.  A nice semi-modern treatment of the matter can be found in 
\cite{Din03}.

\subsection{Getting rid of the circle}\label{sec:proofThm4}

We are still lacking a proof of Theorem~\ref{CGTtangent2} that we presented earlier (see Figure~\ref{fig:tangents}, left): the version of the CQT that relies on three confocal conics and bypasses the circle.
In exactly the same fashion as we did before, this theorem can be rephrased using the absolute circle points  $\mathsf{I}$
and  $\mathsf{J}$ to express the fact that three conics are confocal.
This time, the circle points are not even explicitly needed. Expressing confocality reduces to being tangent to the same four lines. 
{This can be seen as follows: Given a conic $\mathcal{C}$ we can construct its foci by the construction of} 
Figure~\ref{fig:foci}  {on the right: Form the tangents of  $\mathsf{I}$
and  $\mathsf{J}$ to the conic and consider their remaining four intersections. If a second conic has the same foci then it must result in the same tangents. 
Thus, up to projective transformations a collection of conics is confocal if they have four tangents in common. No other Euclidean  terms are relevant in Theorem~\ref{CGTtangent2}, hence we can discard the presence of  $\mathsf{I}$
and  $\mathsf{J}$.}

We now follow the same Cayley–Bacharach mechanism as in the previous section: Dualising $\rightarrow$ indicating a Cayley–Bacharach situation $\rightarrow$ organising hypothesis and conclusion such that no ambiguity arises. 
{Figure~\ref{fig:CBT5}  dualises the configuration of  Theorem~\ref{CGTtangent2}.  Being confocal translates to the fact that $\mathcal{R}$, $\mathcal{G}$ and $\mathcal{X}$ 
pass through the same four points (white). The tangents through $P$ and $Q$ and the tangents through $S$ and $T$ translate into two degenerate conics. 
The four tangents meeting at $M$ corresponds to four collinear (white) points.
The corresponding configuration is shown in Figure~\ref{fig:CBT5} (left). It is a  Cayley--Bacharach configuration with a blue, a red and a green quartic. Hence, three of the single incidences follow from the others. The white points incident with all three conics  translate to the three conics in the original  theorem being confocal. The black points
represent the original tangents. The coincidence of three of the black points with the blue conic corresponds to the conclusion.
The associated  Chasles' quadrilateral is shown on the right.

\begin{figure}[ht]
\centering

\includegraphics[width=.45\textwidth]{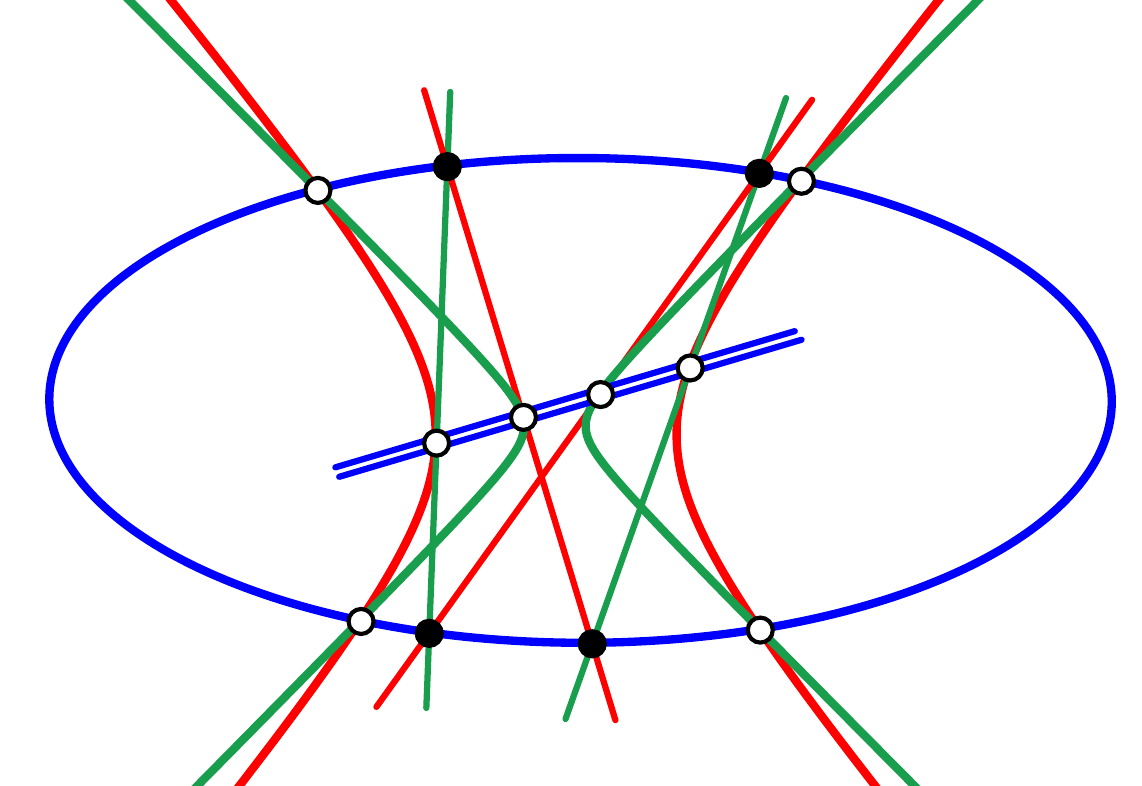}$\qquad$
\includegraphics[width=.36\textwidth]{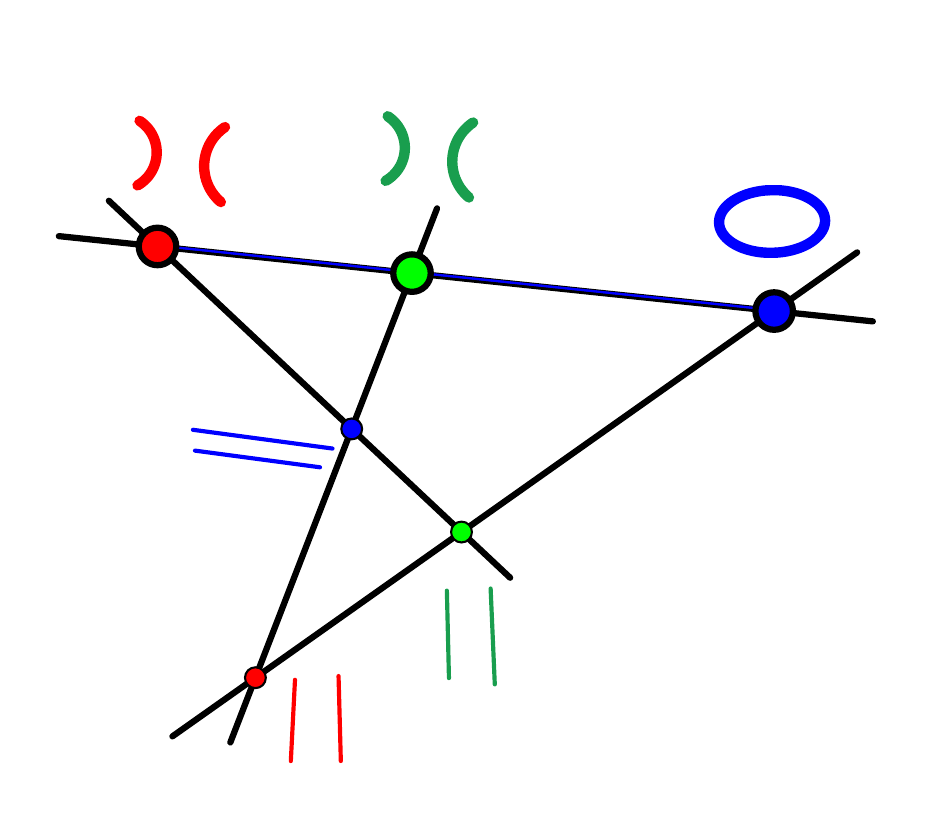}

\begin{picture}(0,0)
\put(40,80){\footnotesize {${\mathcal{R}}$}}
\put(42,62){\footnotesize {$M$}}
\put(75,90){\footnotesize {${\mathcal{G}}$}}
\put(130,70){\footnotesize {${\mathcal{X}}$}}
\put(60,13){\footnotesize {$S$}}
\put(72,14){\footnotesize {$T$}}
\put(82,25){\footnotesize {$P$}}
\put(92,25){\footnotesize {$Q$}}
\put(-100,110){\footnotesize {$S$}}
\put(-40,108){\footnotesize {$T$}}
\put(-90,109){\footnotesize {$P$}}
\put(-50,108){\footnotesize {$Q$}}
\put(-115,10){\footnotesize {${\mathcal{R}}$}}
\put(-132,10){\footnotesize {${\mathcal{G}}$}}
\put(-156,60){\footnotesize {${\mathcal{X}}$}}
\put(-116,50){\footnotesize {$M$}}
\end{picture}
\caption{The situation corresponding to  Theorem~\ref{CGTtangent2}. Cayley--Bacharach interpretations (left), Chasles' quadrilateral in the projective space of planar conics (right). }	
\label{fig:CBT5}
\end{figure}
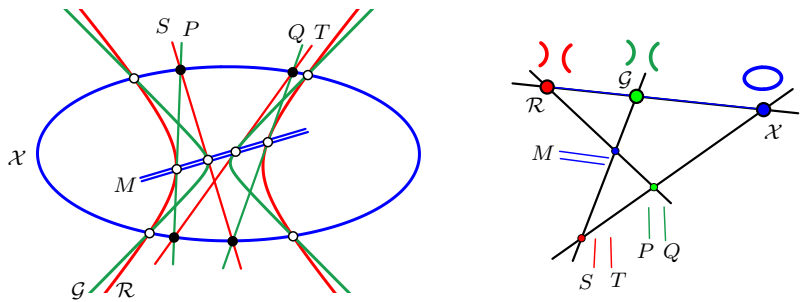
\medskip

 {After this translation},  Theorem~\ref{CGTtangent2} is immediately proved by the Cayley--Bacharach Theorem (or equivalently by the Chasles' quadrilateral argument). 
In fact, the underlying configuration shares exactly the same combinatorics as the Bobenko--Akopyan statement given in Figure~\ref{fig:BobAk}. {We indicated this  by the color coding.}

\section{Conclusion}

The analysis presented in this article clarifies a subtle but persistent ambiguity in many formulations of  Chasles' Quadrilateral Theorem.
The classical statements, beginning with Chasles and Darboux and reappearing in modern literature, often extend a correct equivalence by an additional `moreover' clause that presupposes the uniqueness of a conic in a given pencil. In symmetric configurations, however, two distinct circles can share the same four tangents, and the supplementary claim then fails, since ambiguities arise.
Reformulating the theorem in its projective form sheds light on the deeper underlying reasons for these ambiguities. In resolving the issue, we find that the existence of a circle or conic must be stated relative to the concurrence of the tangents, not the other way around.
Within this perspective,  Chasles' Quadrilateral Theorem becomes a clean corollary of the Cayley--Bacharach theorem.

\bibliographystyle{plain} 
{\small
\bibliography{refsd}

@article{Cha1860,
	author = {Chasles, Michel},
	journal = {Comptes Rendus des s´eances de l'Acad´emie des Sciences},
	pages = {623--633},
	title = {R{\'e}sum{\'e} d'une th{\'e}orie des coniques sph{\'e}riques homofocales.},
	volume = {50},
	year = {1860}}

@article{Cha1843,
	author = {Chasles, Michel},
	journal = {Comptes Rendus hebdomadaires de s{\'e}ances de l'Acad{\'e}mie des sciences},
	pages = {838--844},
	title = {Propri{\'e}t{\'e}s g{\'e}n{\'e}rales des arcs d'une section conique, dont la difference est rectifiable},
	volume = {17},
	year = {1843}}

@article{Bach1886,
	author = {Bacharach, Isaak},
	date = {1886/06/01},
	doi = {10.1007/BF01444338},
	id = {Bacharach1886},
	isbn = {1432-1807},
	journal = {Mathematische Annalen},
	number = {2},
	pages = {275--299},
	title = {Ueber den Cayley'schen Schnittpunktsatz},
	url = {https://doi.org/10.1007/BF01444338},
	volume = {26},
	year = {1886}}

@inbook{Din03,
	author = {Dingeldey, Friedrich},
	language = {ger},
	pages = {6--165},
	publisher = {Springer Fachmedien Wiesbaden GmbH},
	series = {Encyklop{\"a}die der mathematischen Wissenschaften mit Einschluss ihrer Anwendungen},
	title = {Kegelschnitte und Kegelschnittsysteme},
	url = {http://eudml.org/doc/202727},
	volume = {3, T.2, H.1},
	year = {1903}}

@article{BGRGT24a,
	archiveprefix = {arXiv},
	author = {Berman, Leah Wrenn AND G{\'e}vay, G{\'a}bor AND Tabachnikov, Serge AND Richter-Gebert, J{\"u}rgen},
	eprint = {2408.09203},
	journal = {arXiv preprint arXiv: 2408.09203},
	primaryclass = {math.CO},
	title = {When {G}r{\"u}nbaum meets {P}oncelet -- Infinite Classes of Movable $(n_4)$ Configurations},
	url = {https://arxiv.org/abs/2408.09203},
	year = {2024}}

@article{Rey1896,
	author = {Reye, Theodor},
	journal = {Naturforschenden Gesellschaft in Z{\"u}rich},
	number = {2},
	pages = {65--75},
	title = {Beweis einiger {S}{\"a}tze von {C}hasles {\"u}ber konfokale {K}egelschnitte.},
	volume = {14},
	year = {1896}}

@book{ChaGra1841,
	author = {Chasles, Michel AND Graves, Charles},
	publisher = {Dublin: For {G}rant and {B}olton},
	title = {Two Geometrical Memoirs On the General Properties of Cones of the Second Degree, and On the Spherical Conics},
	year = {1841}}

@book{Cha1865,
	author = {Chasles, Michel},
	publisher = {Gauthier-Villars, Paris},
	title = {Trait{\'e} des sections coniques: faisant suite au trait{\'e} de g{\'e}om{\'e}trie sup{\'e}rieure},
	year = {1865}}

@book{AlkZas07,
	author = {Akopyan, Arseniy V. AND Zaslavsky, Aleksej A.},
	publisher = {American Mathematical Society, Providence},
	series = {Mathematical World},
	title = {Geometry of Conics},
	volume = {26},
	year = {2007}}

@article{AkBo18,
	author = {Akopyan, Arseniy V. AND Bobenko, Alexander I.},
	journal = {Trans. Amer. Math. Soc.},
	number = {2825-2854},
	title = {Incircular nets and confocal conics},
	volume = {370},
	year = {2018}}

@book{Dar17,
	author = {Darboux, Gaspard},
	publisher = {Paris, Gauther-Villars},
	title = {Principes de Geometrie Analytique},
	year = {1917}}

@inproceedings{Iz19,
	author = {Izmestiev, Ivan},
	booktitle = {Eighteen essays on non-Euclidean geometry},
	editor = {A. Papadopoulos and V. Alberge},
	pages = {263--318},
	publisher = {European Mathematical Society Publishing House},
	title = {Spherical and hyperbolic conics},
	year = {2019}}

@article{IzTa17,
	author = {Izmestiev, Ivan AND Tabachnikov, Serge},
	journal = {Journal of Integrable Systems},
	pages = {1-36},
	title = {Ivory's theorem revisited},
	volume = {2},
	year = {2017}}

@article{StaB,
	author = {Stachel, Hellmuth},
	journal = {J. Geom.},
	title = {The geometry of billiards in ellipses and their Poncelet grids},
	volume = {112(40)},
	year = {2021}}

@book{Ber87,
	author = {Berger, Marcel},
	publisher = {Springer-Verlag, Berlin},
	title = {Geometry. I. II.},
	year = {1987}}

@article{RiKo99,
	author = {Richter-Gebert, J{\"u}rgen AND Kortenkamp, Ulrich},
	journal = {Springer-Verlag, Berlin},
	title = {The {I}nteractive {G}eometry {S}oftware {C}inderella},
	year = {1999}}

@article{EGH96,
	author = {Eisenbud, David AND Green, Mark AND Harris, Joe},
	journal = {Bulletin American Math. Society},
	pages = {295-324},
	title = {Cayley--{B}acharach theorems and conjecture},
	volume = {33},
	year = {1996}}

@book{RG11,
	author = {Richter-Gebert, J{\"u}rgen},
	publisher = {Springer},
	title = {Perspectives on {P}rojective {G}eometry},
	year = {2011}}

@book{DraRa11,
	author = {Dragovi\'c, Vladimir AND Radnovi\'c, Milena},
	publisher = {Birkh{\"a}user/Springer, Basel},
	title = {Poncelet porisms and beyond. Integrable billiards, hyperelliptic Jacobians and pencils of quadrics},
	year = {2011}}
}

\medskip

\noindent
{\footnotesize
{\sc Department of Mathematics \& Statistics, University of Alaska Fairbanks, USA}\\
Email address: {\tt lwberman@alaska.edu}

\noindent
{\sc Department of Mathematics, Technical University of Munich, Germany}\\
Email address: {\tt richter@tum.de}
}

\end{document}